\documentclass[11pt]{article}
\usepackage{amsfonts}
\usepackage{color}
\usepackage{amsmath,amssymb,comment}
\newtheorem{theo}{Theorem}[section]
\newtheorem{lem}[theo]{Lemma}

\newtheorem{defi}[theo]{Definition}
\newcommand{\proof}{{\sc Proof.} \quad}

\newcommand{\be}{\begin{equation} \label}
	\newcommand{\ee}{\end{equation}}
\newcommand{\bee}{\begin{equation*}}
	\newcommand{\eee}{\end{equation*}}
\newcommand{\bea}{\begin{eqnarray}\label}
	\newcommand{\eea}{\end{eqnarray}}
\newcommand{\bas}{\begin{eqnarray*}}
	\newcommand{\eas}{\end{eqnarray*}}
\newcommand{\bit}{\begin{itemize}}
	\newcommand{\eit}{\end{itemize}}
\newcommand{\ali}{\begin{aligned}}
	\newcommand{\eali}{\end{aligned}}
\newcommand{\spl}{\begin{split}}
	\newcommand{\eespl}{\end{split}}	

\newcommand{\qed}{\hfill$\Box$ \vskip.2cm}
\newcommand{\nn}{\nonumber}
\newcommand{\R}{\mathbb{R}}
\newcommand{\N}{\mathbb{N}}
\newcommand{\pO}{\partial\Omega}

\newcommand{\eps}{\varepsilon}

\newcommand{\wsto}{\stackrel{\star}{\rightharpoonup}}
\newcommand{\hra}{\hookrightarrow}
\newcommand{\io}{\int_\Omega}
\newcommand{\na}{\nabla}
\newcommand{\Del}{\Delta}
\newcommand{\del}{\delta}
\newcommand{\al}{\alpha}
\newcommand{\bet}{\beta}

\newcommand{\lam}{\lambda}

\newcommand{\pa}{\partial}
\newcommand{\bom}{\overline{\Omega}}
\newcommand{\Om}{\Omega}

\newcommand{\wh}{\widehat}

\newcommand{\cd}{\cdot}

\newcommand{\vp}{\varphi}
\newcommand{\lbal}{\left\{ \begin{array}{l}}
	\newcommand{\lball}{\left\{ \begin{array}{ll}}
		\newcommand{\ear}{\end{array} \right.}

	\newcommand{\abs}{\\[5pt]}

	\newcommand{\adb}{\allowdisplaybreaks}
	\newcommand{\tm}{T_{max}}
	\newcommand{\tmn}{T_{max,N}}

	\newcommand{\F}{{\mathcal{F}}}

	\newcommand{\whu}{\wh{u}}
	
\begin{document}
		\adb
		\title{$L^\infty$ blow-up in the fourth-order\\ Jordan-Moore-Gibson-Thompson equation}
		\author{
			Tahir Boudjeriou\footnote{t.boudjeriou@univ-boumerdes.dz}\\
			{\footnotesize Institute of Electrical \& Electronic Engineering}\\
			{\footnotesize University of Boumerdes,  Boumerdes, 35000, Algeria}
			\and
			Filippo Dell'Oro\footnote{filippo.delloro@polimi.it}\\
			{\footnotesize Politecnico di Milano - Dipartimento di Matematica}\\
			{\footnotesize Via Bonardi 9, 20133 Milano, Italy}}
		\date{}
		\maketitle
		\begin{abstract}
			\noindent 
			{We investigate the fourth-order Jordan-Moore-Gibson-Thompson equation
				\begin{equation*}
				u_{tttt}+\al u_{ttt}+\bet u_{tt}-\gamma \Del u_{tt}-\del \Del u_{t}-\varrho \Del u=(f(u))_{ttt}+(f(u))_{tt}
				\end{equation*}
				under homogeneous Dirichlet boundary conditions in a smooth bounded domain $\Om \subset \R^n$ with $n\leq 3$, where
				$\al \in \R$ and $\bet,\gamma,\del,\varrho>0$. First, for any suitably regular initial data and assuming that $f\in C^4(\R)$ with $f(0)=0$, we prove that the corresponding initial-boundary value problem admits a unique maximal strong solution $u$ on $(0, \tm)$.
				Second, by combining three time-integrated formulations with  a series of resolvent-regularized energy estimates, we show that gradient blow-up is ruled out in the sense that
				\begin{equation}\label{EE}
				\tm<\infty \quad\,\, \Rightarrow\quad\,\, \limsup_{t\nearrow \tm}\|u(\cd,t)\|_{L^{\infty}(\Om)}=\infty. \tag{$\star$}
				\end{equation}	
				Finally, assuming  in addition that $f$ is convex and grows superlinearly in the sense that 
				\begin{equation*}
				f''(\xi)\geq 0 \;\; \mbox{ for all}\;\;\xi \in\R, \qquad \frac{f(\xi)}{\xi}\rightarrow +\infty\;\;\mbox{ as}\;\; \xi\rightarrow +\infty , \qquad \int_{\xi_{0}}^{\infty}\frac{d\xi}{f(\xi)}<\infty \;\; \mbox{ for some}\;\; \xi_{0}>0,
				\end{equation*}
				we employ  the extensibility criterion (\ref{EE})  to prove that, for a suitable class of initial data, the solution blows up in finite time in the $L^\infty$-norm. To the best of our knowledge, this provides the first rigorous description of finite-time  blow-up for a genuinely nonlinear equation of fourth order in time.
			}\abs

			\medskip
			\noindent {\bf Key words:} Fourth-order Jordan-Moore-Gibson-Thompson equation; finite-time blow-up; extensibility criterion; nonlinear acoustics.
			
			\smallskip
			\noindent
			{\bf MSC 2020:} 35L15, 35B44, 35L05.
		\end{abstract}
		
		%
		%
		%
		%
		%
		\newpage
		\section{Introduction}\label{intro}
		In a smooth bounded domain $\Om\subset \R^n$ with $n\leq 3$, we consider the initial-boundary value problem for the fourth-order Jordan-Moore-Gibson-Thompson equation
		\be{0}
		\lball
		u_{tttt}+\al u_{ttt}+\bet u_{tt} -\gamma \Del u_{tt}-\del\Del u_{t}-\varrho \Del u= (f(u))_{ttt}+(f(u))_{tt},
		\\[1mm]
		u_{|x\in\pO}=0,\\[1.7mm]
		u(x,0)=u_0(x), \;\;\;u_{t}(x,0)=u_{1}(x), \;\;\;u_{tt}(x,0)=u_{2}(x), \;\;\;u_{ttt}(x,0)=u_{3}(x),
		\ear
		\ee
		where $\al \in \R$ and  $\bet,\gamma,\del,\varrho>0$, {whereas
			$u_0,u_1,u_2,u_3$ are prescribed initial data.} The nonlinear term $f$ is chosen so as to capture, in a natural fashion, the model case 
		\be{IEE1}
		f(u)=k u^2, \quad\,\, k>0,
		\ee 
		but our analysis applies to significantly more general classes of nonlinearities.
		
		\subsection{Background and motivations}
		The modeling of nonlinear acoustics relies, inter alia, on the Jordan-Moore-Gibson-Thompson (JMGT) equation with viscoelastic effects, which plays an important role in the description of high-frequency wave propagation in fluids and gases; see for instance \cite{Lebon1989, Rudenko1977}. Such effects are captured mathematically by memory terms and are physically associated with molecular relaxation, whereby the pressure at a given time depends not only on the instantaneous density but also on its values at preceding times, hence on the history of the process. The mathematical analysis of nonlinear JMGT models with memory has so far focused mainly on semilinear equations involving various types of forcing terms, depending either on the solution itself \cite{Caixeta2016, Las3, Nik4} or on its derivatives \cite{Nik4, Nik5, Nik6}. One such model is the nonlocal JMGT equation
		\be{ITE6}
		\tau u_{ttt}+a u_{tt}-c^2\Delta u-b\Delta u_t
		+\int_0^t g(t-s)\Delta u(s)\,ds
		=\kappa(u^2)_{tt}
		\ee
		where $\tau>0$ denotes the relaxation time, $b>0$ is a diffusivity  parameter, $c^2$ is the squared  speed of sound and $a>0$ represents a friction coefficient. Finally $\kappa=\frac{1}{c^2}\left(\frac{B}{2A}+1\right)$ where $A,B>0$ are parameters associated with the nonlinear interactions.	
		The fourth-order JMGT equation appearing in problem \eqref{0} can be derived from (\ref{ITE6}) when the memory kernel $g$ is the negative exponential,
		which is the most relevant kernel from the physical viewpoint.
		More precisely, 
		for 
		$$g(s)=d e^{-s}, \quad\,\, d\in (0,c^2), $$
		differentiating \eqref{ITE6} with respect to time and adding \eqref{ITE6} we obtain
		$$
		\tau u_{tttt}+(a+\tau)u_{ttt}
		+a u_{tt}
		-b\Delta u_{tt}
		-(c^2+ b)\Delta u_t
		-( c^2-d)\Delta u
		=\kappa(u^2)_{ttt}+\kappa(u^2)_{tt}.
		$$
		Setting 
		$$
		\alpha = \frac{a+\tau}{\tau},\qquad \beta = \frac{a}{\tau},\qquad \gamma = \frac{b}{\tau}, \qquad \delta=\frac{c^2+b}{\tau},
		\qquad \varrho = \frac{c^2-d}{\tau},\qquad k =  \frac{\kappa}{\tau},
		$$		
		we obtain the nonlinear equation appearing in \eqref{0} with $f$ given by \eqref{IEE1}.
		In the linear case
		when $f\equiv0$, the equation reduces to the so-called fourth-order Moore--Gibson--Thompson (MGT) equation
		introduced in \cite{Dell2}
		\be{ITE5}
		u_{tttt}+\al u_{ttt}+\bet u_{tt}-\gamma \Del u_{tt}-\del \Del u_{t}-\varrho \Del u=0.
		\ee 
		In that paper, for $\al, \bet, \gamma, \del,\varrho>0 $ and
		under homogeneous Dirichlet boundary conditions, it has been proved that \eqref{ITE5}
		generates a linear strongly continuous semigroup on the space 
		$W^{1,2}_0(\Omega)\times W^{1,2}_0(\Omega)\times W^{1,2}_0(\Omega)\times L^2(\Omega)$ which is exponentially stable if and only if
		\be{STT}
		\gamma>\frac{\delta}{\alpha}
		\qquad\mbox{and}\qquad
		\left(\beta-\frac{\varrho\alpha}{\delta}\right)>
		-\lambda_1\left(\gamma-\frac{\delta}{\alpha}\right),
		\ee 
		where $\lambda_1>0$ is the first eigenvalue of the Laplace-Dirichlet operator $-\Delta$ on $\Omega$.
		It is important to mention that the fourth-order MGT equation \eqref{ITE5} plays a key role in the study of the stability and instability properties of the linearized version of \eqref{ITE6}, see \cite{Dell3}. In addition, it naturally arises in the context of heat conduction modeling with generalized heat-conduction laws, see for instance 
		\cite{Chand, Dell1, Green1, Green2, Muller, Quint4,Quin,Roy, Strau, Tzou1, Tzou2} and references therein.

		\subsection{Earlier literature}	
		The analysis of finite-time blow-up for solutions to nonlinear JMGT equations has received great attention in recent literature.		For the classical JMGT equation with homogeneous Dirichlet boundary conditions
		\be{ITE9}
		\lball
		\tau u_{ttt}+\alpha u_{tt}
		-\beta\Delta u_t-\gamma\Delta u=k(u^2)_{tt},\\
			u_{|x\in\pO}=0,
		\ear
		\ee
		finite-time blow-up of some solutions, in the sense of unboundedness of the $L^\infty$-norm of $u$, was established in \cite{Nik3} for $\tau,\beta,\gamma>0$ and $\alpha\in\mathbb R$, where the quadratic nonlinearity \eqref{IEE1} is replaced by a general convex superlinearly growing function $f\in C^3(\mathbb R)$. In that article, after establishing local existence and uniqueness of strong solutions together with a sharp extensibility criterion, the Kaplan first-eigenvalue method \cite{PQUI} was used to derive the finite-time blow-up result.	
		In \cite{Chen2026} the nonlinear term $k(u^2)_{tt}$ is replaced (up to a multiplicative constant) by
		$\partial_t\left(|u_t|^2+|\nabla u|^2\right)$
		and finite-time blow-up, together with upper bounds for the lifespans of solutions, is obtained for the corresponding Cauchy problem in $\R^n$. In \cite{Chen1} the nonlinear term $|u_t|^p$ with $1<p\leq (n+1)/(n-1)$ is considered, and finite-time blow-up of solutions is shown. The nonlinearity $|u|^p$ is treated in \cite{Chen2,Chen2021}, where nonexistence of global solutions is proved under suitable assumptions on $p$. Related finite-time blow-up results for the Cauchy problem associated with the Kuznetsov equation in $\mathbb{R}^n$ were obtained in \cite{DeK}. On the other hand, global existence for suitably small initial data was established in bounded domains in \cite{Kaltenbacher2012} and for the Cauchy problem in the whole space in \cite{Racke2021, SaidHouari2022}. The asymptotic behavior as $\tau\searrow0$ was investigated in \cite{Bon1, Bon2, Kaltenbacher2019} and further singular limits have been studied in \cite{KaltenbacherNikolic2021}. We also refer to \cite{Lizama2019}, where local null controllability results have been obtained.
		More recently, a model coupling thermal dynamics with Moore--Gibson--Thompson-type equations involving third-order in time derivatives in a one-dimensional viscoelastic medium of standard linear solid type was introduced in \cite{Cla}. Its mathematical analysis was carried out in \cite{Black, Cla, Frik}, including local existence and uniqueness of strong solutions, local existence of classical solutions, as well as global existence and exponential stabilization for small initial data. 	See also \cite{Benabbas2026}, where  local existence of solutions is proved for suitably small initial data.

		\subsection{Main results} 
		
		 Evolution equations of fourth order in time, such as  (\ref{0}), appear to be considerably more difficult to analyze than their lower-order counterparts, even at the linear level, due to the delicate stability issues associated with their temporal structure. The presence of nonlinear terms introduces further difficulties, which may account for the apparent lack of results concerning finite-time blow-up in this setting.
	This paper attempts to bridge this gap by analyzing finite-time blow-up for strong solutions to the fourth-order Jordan-Moore-Gibson-Thompson problem~\eqref{0}. 
		Our aim here is not only to establish finite-time blow-up in the $L^\infty$-norm, {but also to rule out the occurrence of gradient blow-up.
			Recall that \textit{gradient blow-up} refers to the situation in which the solution remains bounded in $L^\infty(\Omega)$ whereas its spatial gradient blows up as $t\nearrow \tm$, namely
			\be{Grad}
			\tm<\infty, \qquad
			\sup_{0\leq t< \tm}\|u(\cdot,t)\|_{L^\infty(\Omega)}<\infty
			\quad\,\, \mbox{ and }\quad\,\, 
			\lim_{t\nearrow \tm}
			\|\nabla u(\cdot,t)\|_{L^\infty(\Omega)}=\infty,
			\ee
			see for instance \cite{Fila1994, Souplet2002}. It is worth mentioning that
			the occurrence of gradient blow-up phenomena is not only of mathematical interest but also of physical relevance. For example, in  therapeutic ultrasound  the formation of steep gradients can be exploited in applications such as shock-wave  lithotripsy \cite{Cleveland2012}. Investigation of blow-up phenomena in nonlinear acoustic models is therefore important not only for understanding the mathematical behavior thereof, but also for clarifying the possible physical implications. In the third-order JMGT model (\ref{ITE9}) gradient finite-time blow-up does not occur, as shown recently 
			in \cite{Nik3}. Our arguments here are inspired by those in \cite{Nik3}, but the techniques developed there cannot be applied directly to the present setting and the presence of the fourth-order in time derivative in (\ref{0}) makes the analysis very delicate.  
			The passage from the third-order problem studied in~\cite{Nik3} to the present fourth-order equation is not a formal shift of the energy method. Indeed, the additional temporal derivative changes the nonlinear structure through the appearance of 
			\[ (f(u))_{ttt} = f'(u)u_{ttt} +3f''(u)u_tu_{tt} +f'''(u)u_t^3, \]
which involves new nonlinear interactions among lower-order time derivatives.
These features require suitable nontrivial modifications of the functionals and differential inequalities employed in \cite{Nik3}, which allow us to establish the local existence and the refined extensibility criterion as well as to derive the finite-time blow-up. Notably, our results are valid for arbitrary physically meaningful choices of the parameters $\alpha,\beta,\gamma,\delta,\varrho$, without requiring any relation of the form \eqref{STT}, and for a broad class of nonlinearities $f$. Our first main result reads as follows.
			\begin{theo}\label{Th1}
				Suppose that $\al \in \R$ and $ \bet,\gamma,\del,\varrho>0$, as well as
				\be{f}
				f\in C^{4}(\R) \quad\,\, \mbox{with}\quad\,\,  f(0)=0.
				\ee 
				Then, for any initial data satisfying
				\be{IN0}
				u_{0},u_{1}\in W^{4,2}(\Om)\cap W_{0}^{1,2}(\Om),\quad\,\, u_{2}, u_{3}\in W^{2,2}(\Om)\cap W_{0}^{1,2}(\Om),
				\ee 
				there exists $\tm \in (0, \infty]$ 
				and a uniquely determined function $u$ with regularity 
				\be{T1}
				\lbal
				u\in C^{0}([0,\tm);W_{0}^{1,2}(\Om)\cap C^{0}(\bom))\cap L_{loc}^{\infty}([0,\tm); W^{2,2}(\Om)), \\
				u_{t}\in C^{0}([0,\tm);W_{0}^{1,2}(\Om)\cap C^{0}(\bom))\cap L_{loc}^{\infty}([0,\tm); W^{2,2}(\Om)), \\
				u_{tt}\in C^{0}([0,\tm);W_{0}^{1,2}(\Om)\cap C^{0}(\bom))\cap L_{loc}^{\infty}([0,\tm); W^{2,2}(\Om)), \\
				u_{ttt}\in C^{0}([0,\tm);L^2(\Om))\cap L_{loc}^{\infty}([0,\tm); W_{0}^{1,2}(\Om)), \\
				u_{tttt}\in L_{loc}^{\infty}([0,\tm); L^2(\Om)),
				\ear 
				\ee
				such that $u$ is a strong solution of \eqref{0} in the sense of Definition \ref{Deff} below. Moreover, the following blow-up criterion holds: 
				\be{BU}
				\tm<\infty\quad\,\,\Rightarrow\quad\,\, \limsup_{t\nearrow \tm} \|u(\cd,t)\|_{L^{\infty}(\Om)}=\infty.
				\ee 
			\end{theo}
			To formulate our second main result, let $(e_{i})_{i\geq 1}\subset C^{\infty}(\bom)$ be the eigenfunctions of
			the Laplace-Dirichlet operator $-\Del$
			$$
			\lball
			-\Del e_{i}=\lam_{i}e_{i}, \\[0.5mm]
			{e_{i}}_{|\pa\Om}=0,
			\ear
			$$
			normalized so that $\io e_{i} e_{j}=\del_{ij}$ for all $i,j\geq 1$, and such that
			$$
			e_{1}>0 \quad\, \mbox{ in }\, \Om.
			$$
			\begin{theo}\label{Th1.3}
				Assume that $f$ satisfies (\ref{f}) and the following conditions: 
				\be{T1.3.1}
				\,f''(\xi)\geq 0 \qquad \mbox{for all }\; \xi\in \R,
				\ee 
				\be{T1.3.2}
				\frac{f(\xi)}{\xi}\rightarrow +\infty \qquad \mbox{as }\; \xi \to +\infty
				\ee
				and
				\be{T1.3.3}
				\int_{\xi_{0}}^{\infty}\frac{d\xi}{f(\xi)}<\infty \qquad \mbox{for some }\; \xi_{0}>0.
				\ee 
				Then, for every $T_{0}>0$ there exists $K_{0}:=K_{0}(T_{0})$ such that if $u_{0}\in W^{4,2}(\Om)\cap W_{0}^{1,2}(\Om)$ fulfils 
				\be{T1.3.4}
				\io u_{0}e_{1}>K_{0}
				\ee 
				then one can find $K_{1}:=K_{1}(u_{0})>0$ with the property that for every $u_{1}\in W^{4,2}(\Om)\cap W_{0}^{1,2}(\Om)$ satisfying
				\be{T1.3.5}
				\io u_{1}e_{1}>K_{1}
				\ee 
				there exists  $K_{2}:=K_{2}(u_{1})>0$ such that whenever $u_{2}\in W^{2,2}(\Om)\cap W_{0}^{1,2}(\Om)$  obeys
				\be{T1.3.6}
				\io u_{2}e_{1}>K_{2}
				\ee 
				there exists $K_{3}:=K_{3}(u_{2})>0$ such that for any $u_{3}\in W^{2,2}(\Om)\cap W_{0}^{1,2}(\Om)$ with  
				\be{T1.3.7}
				\io u_{3}e_{1}>K_{3}
				\ee 
				the corresponding strong solution $u$ of (\ref{0}) blows up no later than time $T_{0}$. More precisely, in the sense of Theorem \ref{Th1}, its maximal existence time satisfies $\tm\leq T_{0}$ and 
				\be{T1.3.55}
				\limsup_{t\nearrow \tm} \|u(\cd,t)\|_{L^{\infty}(\Om)}=\infty.
				\ee 
			\end{theo}
			
			\subsection{Proof strategy} By means of integrated versions of \eqref{0} (see Lemma~\ref{lem2.3}), we first establish uniqueness for the strong solution introduced in Definition \ref{Deff} under the assumption that $f$ satisfies \eqref{f}
			(see Lemma~\ref{lem2.4}). Next, starting from the approximate solutions constructed in Lemma \ref{lem1.1} by means of a Galerkin approximation, we show that (upon dropping the approximation index $N$)
			$$
			\F(t)\leq c_{6}+c_{18}t
			\qquad\mbox{for all }t\in(0,T_N),
			$$
			where $T_{N}$ is defined in (\ref{TNNN}), $c_{6}$ and $c_{18}$ are defined in \eqref{E13.10} and \eqref{E13.14}, respectively, while
			\bee{} 
			\spl \F(t):=&\, \frac{1}{2}\io |\na u_{ttt}(\cd,t)|^2+\frac{\gamma}{2}\io |\Del u_{tt}(\cd,t)|^2+\del \io \Del u_{t}(\cd,t)\Del u_{tt}(\cd,t) +\varrho \io \Del u(\cd,t)\Del u_{tt}(\cd,t)\\[1.3mm] &+\frac{8 \del^2}{3\gamma}\io |\Del u_{t}(\cd,t)|^2+\frac{8 \varrho^2}{3\gamma} \io |\Del u(\cd,t)|^2.
		\end{split} 
		\eee
		A self-map argument together with the energy estimates and a compactness argument then yields the existence of a local strong solution in the sense of Definition \ref{Deff}, provided that the initial data satisfy the regularity assumptions in \eqref{IN}. The solution can subsequently be extended as long as
		\bee{} 
		\tm<\infty\, \Rightarrow\, \limsup_{t\nearrow \tm}\left[\|\na u_{ttt}(\cd,t)\|_{L^2(\Om)}+\|\Del u_{tt}(\cd,t)\|_{L^2(\Om)}+\|\Del u_{t}(\cd,t)\|_{L^2(\Om)}+\|\Del u(\cd,t)\|_{L^2(\Om)}\right]=\infty. \eee
		The crucial step in the proof of Theorem \ref{Th1} is to sharpen this extensibility criterion to
		$$
		\tm<\infty
		\quad\,\,\Rightarrow\quad\,\,
		\limsup_{t\nearrow \tm}\|u(\cdot,t)\|_{L^\infty(\Omega)}=\infty,
		$$
		thereby ruling out gradient blow-up. This is achieved by means of the resolvent-type regularization introduced in Lemma \ref{lem1.8}, together with suitable higher-order energy functionals; see Lemmas~\ref{lem1.9}-\ref{lem1.10} and \ref{lem2.11}. The use of these functionals requires additional regularity on the initial data, expressed by the assumptions in \eqref{IN0}.	
		Finally, assuming in addition to \eqref{IN0} that the initial data satisfy \eqref{T1.3.4}--\eqref{T1.3.7} and that $f$ is convex and superlinear in the sense of \eqref{T1.3.1}--\eqref{T1.3.3}, we prove Theorem~\ref{Th1.3} by combining Kaplan's first-eigenvalue method with a self-map argument.
		
		\subsection*{A word of warning}
		Unless explicitly stated otherwise, we always assume that $\Om\subset \R^n$ is a smooth bounded domain
		with $n\leq 3$, that $\al \in \R$ and that $\bet,\gamma,\del,\varrho>0$. Moreover, we assume that (\ref{f}) is satisfied.
		\section{Preliminaries and uniqueness of strong solutions} 
		The objective of this section is twofold. First, we state some preliminary results that will be invoked later on. Second, we establish the uniqueness of strong solutions whenever such solutions exist. The analysis presented here relies exclusively on the following regularity assumptions
		\be{IN}
		u_{0}, u_{1}, u_2\in W^{2,2}(\Om)\cap W_{0}^{1,2}(\Om) \qquad\text{and}\qquad u_{3}\in W_{0}^{1,2}(\Om).
		\ee 
		Compared with the assumptions on $u_{0}$, $u_{1}$ and $u_{3}$ from (\ref{IN0}), these conditions are weaker. 
		
		\medskip	
		We start by introducing the notion of strong solvability that will be employed throughout this paper.   
		\begin{defi}\label{Deff}
			Suppose  $\al \in \R$, $\bet,\gamma,\del,\varrho>0$ and $f\in C^3(\R) $, as well as $T\in (0,\infty]$. A function $u$ satisfying the regularity 
			\be{Def1}
			\lbal
			u\in C^{0}([0,T);W_{0}^{1,2}(\Om)\cap C^{0}(\bom))\cap L_{loc}^{\infty}([0,T); W^{2,2}(\Om)), \\
			u_{t}\in C^{0}([0,T);W_{0}^{1,2}(\Om)\cap C^{0}(\bom))\cap L_{loc}^{\infty}([0,T); W^{2,2}(\Om)), \\
			u_{tt}\in C^{0}([0,T);W_{0}^{1,2}(\Om)\cap C^{0}(\bom))\cap L_{loc}^{\infty}([0,T); W^{2,2}(\Om)), \\
			u_{ttt}\in C^{0}([0,T);L^2(\Om))\cap L_{loc}^{\infty}([0,T); W_{0}^{1,2}(\Om)), \\
			u_{tttt}\in L_{loc}^{\infty}([0,T); L^2(\Om)),
			\ear 
			\ee 
			is called a strong solution of (\ref{0}) in $\Om \times (0,T)$ if
			\bee
			u(\cd,t)=u_{0}, \qquad u_{t}(\cd,0)=u_{1}, \qquad u_{tt}(\cd,0)=u_{2}, \qquad u_{ttt}(\cd,0)=u_{3},
			\eee 
			and the following identity holds a.e.\ in $\Om\times (0,T)$
			\be{Def1E}
			u_{tttt}+\al u_{ttt}+\bet u_{tt}-\gamma \Del u_{tt}-\del \Del u_{t}-\varrho \Del u=f'(u)u_{ttt}+3f''(u)u_{t}u_{tt}+f'''(u)u^3_{t}+f'(u)u_{tt}+f''(u)u^2_{t}.\,\,\,
			\ee 
		\end{defi}

		\medskip
		The following lemma will be employed  in Lemmas \ref{lem1.9}-\ref{lem1.10}, and \ref{lem2.11} to overcome the lack of regularity of strong solutions in order to sharpen the extensibility criterion appearing in Lemma \ref{lem2.5} to that in (\ref{BU}). It collects several properties of a resolvent-type regularization of functions from $L^2(\Om)$. Its proof can be found in \cite[Lemma 3.2]{Nik3}.
		\begin{lem}\label{lem1.8}
			Let $A$ be the realization of $-\Del $ in $L^2(\Om)$ with domain  $D(A)=W^{2,2}(\Om)\cap W_{0}^{1,2}(\Omega)$. Then, for every $\eps >0$, the resolvent operator $(1+\eps A)^{-1}$ satisfies  
			\be{P1}
			\Del (1+\eps A)^{-1}\vp = (1+\eps A)^{-1}\Del \vp \quad \mbox{ for all }\; \vp \in D(A)
			\ee 
			and 
			\be{P2}
			\io \left|(1+\eps A)^{-1}\vp \right|^2\leq \io \vp^2\quad \mbox{ for all }\; \vp \in L^2(\Om).
			\ee 
			In addition 
			\be{P3}
			\io \left|\na (1+\eps A)^{-1}\vp \right|^2\leq \io |\na \vp|^2\quad \mbox{ for all }\; \vp \in W_{0}^{1,2}(\Om).
			\ee 
			Finally, for every fixed $\vp \in L^2(\Om)$, we have 
			\be{P4}
			(1+\eps A)^{-1}\vp \rightarrow \vp \quad \mbox{in} \; L^2(\Om)\quad \mbox{ as }\; \eps\searrow 0.
			\ee 
		\end{lem}
		The three elementary integrated forms of (\ref{0}), stated in the next lemma, will play a crucial role not only in establishing uniqueness and sharpening the extensibility criterion (\ref{E14.1}) to the form in (\ref{BU}), but also in the proof of the finite-time blow-up result in Theorem \ref{Th1.3}. The appearance of $u_{3} $, $\Del u_{1}$ and $\Del u_{0}$ in (\ref{E15.5}) explains the need for the additional 
		assumptions in (\ref{IN0}) beyond those imposed in (\ref{IN}); Indeed, Lemmas \ref{lem1.9}-\ref{lem1.10} explicitly rely on $L^2$-bounds for the quantities $\Del \psi_{1}$, $\Del \psi_{2}$ and $\Del \psi_{3}$, defined in (\ref{E15.2}), (\ref{E15.3}) and (\ref{E15.4}), respectively. 
		\begin{lem}\label{lem2.3}
			Assume that (\ref{IN}) holds and let $u$ be  a strong solution of (\ref{0}) in $\Om \times (0,T)$ for some $T\in (0, \infty]$. Let also
			\be{E15.1}
			\ali
			v(x,t)&:= \int_{0}^{t}u(x,s)\,ds,\\ w(x,t)&:=\int_{0}^{t}\int_{0}^{s}u(x,\sigma)\,d\sigma ds,\\ 
			z(x,t)&:=\int_{0}^{t}\int_{0}^{s}\int_{0}^{\sigma}u(x,\xi)\,d\xi d\sigma ds,
			\eali 
			\ee 
			for all $(x,t)\in \Om \times (0,T)$, and set 
			\be{E15.01}
			F(u(x,t)):=\int_{0}^{t}f(u(x,s))\,ds, \qquad (x,t)\in \Om\times (0,T).
			\ee 
			Moreover, define 
			\begin{align}
				\psi_{1}(x):=&\,u_{3}(x)+\al u_{2}(x)+\bet u_{1}(x)-\gamma \Del u_{1}(x)-\del \Del u_{0}(x)\nn\\&-f''(u_{0}(x))u^2_{1}(x)-f'(u_{0}(x))u_{2}(x)-f'(u_{0}(x))u_{1}(x),\label{E15.2}\\
				\psi_{2}(x)&:=u_{2}(x)+\al u_{1}(x)+\bet u_{0}(x)-\gamma \Del u_{0}(x)-f'(u_{0}(x))u_{1}(x)-f(u_{0}(x)),\label{E15.3}\\
				\psi_{3}(x)&:=u_{1}(x)+\al u_{0}(x)-f(u_{0}(x)),\label{E15.4}
			\end{align}
			for all $x\in \Om$. Then the following identities hold a.e.\ in $\Omega \times (0,T)$: 
			\be{E15.5}
			u_{ttt}+\al u_{tt}+\bet u_{t}-\gamma \Del u_{t}-\del \Del u-\varrho \Del v= f''(u)u^2_{t}+f'(u)u_{tt}+f'(u)u_{t}+\psi_{1},
			\ee 
			\be{E15.6}
			u_{tt}+\al u_{t}+\bet u-\gamma \Del u-\del \Del v-\varrho \Del w= f'(u)u_{t}+f(u)+t\psi_{1}+\psi_{2},
			\ee 	
			and 
			\be{E15.07}
			u_{t}+\al u+\bet v-\gamma \Del v-\del \Del w-\varrho \Del z= f(u)+F(u)+\frac{t^2}{2}\psi_{1}+t\psi_{2}+\psi_{3}.
			\ee 
		\end{lem}
		\proof
		The identities follow immediately  by successively integrating  in  time (\ref{Def1E}).
		\qed
		We conclude this section by proving the uniqueness property asserted in Theorem \ref{Th1}. We shall make use of our overall assumption $n\leq 3$ for the first time. 
		\begin{lem}\label{lem2.4}
			Assume that (\ref{IN}) holds. Then for every $T\in (0,\infty]$ problem (\ref{0}) admits at most one strong solution in $\Om \times (0,T)$.
		\end{lem}
		\proof
		Let $u$ and $\whu$ be two strong solutions of (\ref{0}) in $\Om \times (0,T)$, and fix $T_{0}\in (0,T)$. By the local Lipschitz continuity of $f$ and the boundedness of $u$ and $\whu$ in $\Om \times (0,T_{0})$, which follows from the continuous embedding $H^2(\Om)\hookrightarrow L^{\infty}(\Om)$, there exists $c_{1}=c_{1}(T_{0})>0$ such that
		\be{fL}
		|f(u)-f(\whu)|\leq c_{1}|u-\whu| \qquad \mbox{in }\, \Om\times (0,T_{0}).
		\ee
		For $(x,t)\in \bom \times [0,T_{0}]$, define 
		$$ U(x,t):=u(x,t)-\whu(x,t)\qquad \mbox{and}\qquad V(x,t):=\int_{0}^{t}u(x,s)ds-\int_{0}^{t}\whu(x,s)ds,$$ 
		$$W(x,t):=\int_{0}^{t}\int_{0}^{s}u(x,\sigma)\,d\sigma ds-\int_{0}^{t}\int_{0}^{s}\whu(x,\sigma)\,d\sigma ds,$$ 
		as well as
		$$Z(x,t):=\int_{0}^{t}\int_{0}^{s}\int_{0}^{\sigma}u(x,\xi)\,d\xi d\sigma ds -\int_{0}^{t}\int_{0}^{s}\int_{0}^{\sigma}\whu (x,\xi)\,d\xi d\sigma ds. $$ 
		In view of (\ref{E15.01}) and (\ref{fL}), an application of Young's and Cauchy-Schwarz inequalities yields 
		\be{FT}
		U\cdot \{F(u)-F(\whu)\}\leq c_{1}|U|\left\{\int_{0}^{t}|U|\right\}\leq c_{1}U^2+c_{1}T_{0}\int_{0}^{t}U^2 \quad\,\, \mbox{in }  \Om \times (0,T_{0}).
		\ee 
		Then, invoking (\ref{Def1}) with  $\mu=\frac{8\del^2}{3\gamma}$ and $\lambda=\frac{8\varrho^2}{3\gamma}$, the functional
		\bee
		\spl
		\F(t):=&\frac{1}{2}\io U^2(\cd,t)+\frac{\gamma}{2}\io |\na V(\cd,t)|^2+\del \io \na V(\cd,t)\cd \na W(\cd,t)+\varrho \io \na V(\cd,t)\cd \na Z(\cd,t)\\
		&+\mu \io |\na W(\cd,t)|^2+\lambda \io |\na Z(\cd,t)|^2
	\end{split} 
	\eee 
	belongs to $C^1([0,T_{0}])$. Differentiating $\F$ with respect to time, integrating by parts and using (\ref{E15.07})-(\ref{FT}) together with Young's inequality, we arrive at 
	\allowdisplaybreaks
	\begin{align*}
		\frac{d}{dt}\F(t)=&\io UU_{t}+\gamma \io \na U\cd \na V+\del \io \na U\cd \na W+\del \io |\na V|^2+\varrho \io \na U\cd \na Z\nn\\
		&+\varrho \io \na V\cd \na W+2\mu \io \na V\cd \na W+2\lambda \io \na W\cd \na Z\nn\\
		=&\io U\left\{U_{t}-\gamma\Del V-\del \Del W-\varrho \Del Z\right\}+\del \io |\na V|^2\nn\\
		&+ \varrho \io \na V\cd \na W+2\mu \io \na V\cd \na W+2\lambda \io \na W\cd \na Z\\
		\leq & -\al \io U^2-\frac{\bet}{2}\frac{d}{dt}\io V^2+\io U\cd\big\{f(u)-f(\whu)+F(u)-F(\whu)\big\}+\del \io |\na V|^2\\&
		+\varrho \io |\na V|^2\nn+\varrho \io |\na W|^2+\mu \io |\na V|^2+\mu \io |\na W|^2+\lambda \io |\na W|^2+\lambda \io |\na Z|^2\nn\\
		\leq & \big[|\al|+2c_{1}\big]\io U^2-\frac{\bet}{2}\frac{d}{dt}\io V^2+c_{1}T_{0}\int_{0}^{t}\io U^2+\del \io |\na V|^2+\varrho \io |\na V|^2
		\nn\\&+\varrho \io |\na W|^2+\mu \io |\na V|^2+\mu \io |\na W|^2+\lambda \io |\na W|^2+\lambda \io |\na Z|^2 \nn
	\end{align*}
	for all $t\in (0, T_{0})$. Young's inequality together with the choices of $\mu$ and $\lambda$ yield  
	$$ \F(t)\geq \frac{1}{2}\io U^2+\frac{\gamma}{4} \io |\na V|^2+\frac{\mu }{4}\io |\na W|^2+\frac{\lambda}{4}\io |\na Z|^2$$
	for all $t\in (0,T_{0})$. Combining the two inequalities above we obtain 
	$$
	\frac{d}{dt}\F(t)+\frac{\bet}{2}\frac{d}{dt}\io V^2\leq c_{2}\F(t)+2c_{1}T_{0}\int_{0}^{t}\F(s)\,ds
	$$
	where $c_{2}:=2\cd\{|\al|+2c_{1}\}+\frac{(\del+\varrho+\mu)4}{\gamma}+\frac{4(\varrho+\mu+\lambda)}{\mu}+4$. Integrating the latter inequality with respect to time and using $\F(0)= 0$ and $\io |V (\cd,0)|^2= 0$, we infer that
	\bee
	\spl
	\F(t)+\frac{\bet}{2}\io V^2+\frac{c_{2}\bet}{2}\int_{0}^{t}e^{c_{2}(t-s)}\io V^2\,ds&\leq 2c_{1}T_{0}\int_{0}^{t}e^{c_{2}(t-s)}\int_{0}^{s}\F(\sigma)\,d\sigma ds\\
	&\leq 2c_{1}T^2_{0}e^{c_{2}T_{0}}\int_{0}^{t}\F(s)\,ds, \qquad \mbox{for all } t\in (0,T_{0}).
\end{split}
\eee 
An application of Gronwall's lemma implies $\F\equiv 0$, and therefore $U\equiv 0$ in $\Om \times (0,T_{0})$. Since $T_{0}\in (0,T)$ is arbitrary, letting  $T_{0}\nearrow T$ gives $u\equiv\whu$ in $\Om \times (0,T).$
\qed
\section{Local existence and extensibility criterion}
The aim of this section is,  on the one hand, to establish that, for all suitable regular initial data fulfilling (\ref{IN}), local strong solutions to (\ref{0}) can  always be constructed. On the other hand, the main result of this section Lemma \ref{lem2.5} provides an extensibility criterion which will be a key ingredient in ruling out gradient blow-up phenomena in Section \ref{RO}.  

\smallskip
In what follows, we establish the existence and extensibility for the approximate solutions. To this end, we employ	a Galerkin-type discretization to construct approximate solutions  to (\ref{0}). We begin by introducing the projected initial data $u_{0N}, u_{1N}, u_{2N}$, and $u_{3N}$ associated, respectively, with the given initial data $u_{0}$, $u_{1}$, $u_{2}$, and $u_{3}$, satisfying (\ref{IN}), by setting 
\be{E1}
u_{kN}(x):=\sum_{i=1}^{N}\left\{\io u_{k}e_{i}\right\}\cd e_{i}(x),\qquad x\in \bom, \quad N\in \N,
\ee
for $k\in\{0,1,2, 3\}$. 
We also introduce the finite-dimensional subspaces 
\bee
V_{N}:=\mbox{span}\{e_{1}, \ldots, e_{N}\}\subset W^{2,2}(\Om)\cap W_{0}^{1,2}(\Om).
\eee

\begin{lem}\label{lem1.1}
	Let $k\in \{0,1,2,3\}$ and $(u_{kN})_{N\in \N}$ be as in (\ref{E1}). Then, for each $N\in \N$, there exist $\tmn\in (0, \infty]$ and a function
	$
	u_{N}\in C^4([0, \tmn); V_{N})
	$
	satisfying
	\be{E4}
	u_{N}(\cd,0)=u_{0N}, \quad u_{Nt}(\cd,0)=u_{1N}, \quad u_{Ntt}(\cd, 0)=u_{2N}, \quad u_{Nttt}(\cd,0)=u_{3N},
	\ee 
	and such that, for all $\vp\in C^0([0, \tmn); V_{N})$, there holds 
	\be{E5}
	\ali
	&\io u_{Ntttt}\vp +\al \io u_{Nttt}\vp +\bet \io u_{Ntt}\vp \\
	&=\gamma \io \Del u_{Ntt}\vp+\del \io \Del u_{Nt}\vp+\varrho \io \Del u_{N}\vp+\io (f(u_{N}))_{ttt}\vp +\io (f(u_{N}))_{tt}\vp
	\eali
	\ee
	for all $t\in (0, \tmn)$. Moreover, the following extensibility criterion holds: 
	\be{E6}
	\ali
	&\mbox{ if} \;\; \tmn<\infty\, \mbox{ then}\\[0.7mm]
	& \limsup_{t\nearrow \tmn} \left[\|u_{N}(\cd,t)\|_{L^2(\Om)}+\|u_{Nt}(\cd,t)\|_{L^2(\Om)}+\|u_{Ntt}(\cd,t)\|_{L^2(\Om)}+\|u_{Nttt}(\cd,t)\|_{L^2(\Om)}\right]=\infty.
	\eali
	\ee 
\end{lem}
\proof
In a standard way (\ref{E4})-(\ref{E5}) can be transformed into an initial-value problem for a  system of $N$ fourth-order ordinary differential equations  with (locally) Lipschitz continuous source. In view of the equivalence of the norms on the  finite-dimensional space $V_{N}$, the existence result and the extensibility criterion in (\ref{E6}) follow directly from the Picard-Lindel\"of theorem.
\qed
The approximation constructed above fits naturally within the energy functional introduced in (\ref{E7}) below, at least over sufficiently short time intervals. This feature is a consequence of the preliminary observation given below, which will play a key role in Lemma \ref{lem1.3}.
\begin{lem}\label{lem1.2}
	Assume that (\ref{IN}) holds and define 
	\be{E07}
	\mu:=\frac{8 \del^2}{3\gamma} \qquad \mbox{and} \qquad \lam:= \frac{8 \varrho^2}{3\gamma}.
	\ee
	For $N\in \N$, we introduce the functional $\F_{N}\in C^1([0,\tmn))$ given by
	\be{E7}
	\spl
	\F_{N}(t):=&\, \frac{1}{2}\io |\na u_{Nttt}(\cd,t)|^2+\frac{\gamma}{2}\io |\Del u_{Ntt}(\cd,t)|^2+\del \io \Del u_{Nt}(\cd,t)\Del u_{Ntt}(\cd,t)\\
	&+\varrho \io \Del u_{N}(\cd,t)\Del u_{Ntt}(\cd,t)+\mu\io |\Del u_{Nt}(\cd,t)|^2+\lambda \io |\Del u_{N}(\cd,t)|^2.
\end{split}
\ee
This functional obeys
\be{E8}
\F_{N}(t)\geq \frac{1}{2}\io |\na u_{Nttt}(\cd,t)|^2+\frac{\gamma}{4}\io |\Del u_{Ntt}(\cd, t)|^2+\frac{\mu}{4}\io |\Del u_{Nt}(\cd,t)|^2+\frac{\lambda}{4}\io |\Del u_{N}(\cd, t)|^2,
\ee 
while its derivative satisfies
\begin{align}
\label{Def1EE}
\frac{d}{dt}\F_{N}(t)\leq& \left[\|f'(u_{N})\|_{L^{\infty}(\Om)}+|\alpha|+11\right] \io |\na u_{Nttt}|^2+\bet^2\io |\na u_{Ntt}|^2+(\mu+\del +1)\io |\Del u_{Ntt}|^2\nonumber\\
&+(\mu+\lambda+\varrho^2)\io |\Del u_{Nt}|^2+\lambda\io |\Del u_{N}|^2+ \|f''''(u_{N})\|^2_{L^{\infty}(\Om)}\|u_{Nt}\|_{L^{\infty}(\Om)}^{6}\|\na u_{N}\|^2_{L^{2}(\Om)}\nonumber\\
&+9\|f'''(u_{N})\|^2_{L^{\infty}(\Om)}\cd \left[\|u_{Nt}\|^{4}_{L^{\infty}(\Om)}\|\na u_{Nt}\|^2_{L^2(\Om)}+\|u_{Nt}\|^2_{L^{\infty}(\Om)}\|u_{Ntt}\|_{L^{4}(\Om)}^2\|\na u_{N}\|^2_{L^{4}(\Om)}\right.\nonumber\\
&\left.\quad+\|u_{Nt}\|^4_{L^{\infty}(\Om)}\|\na u_{N}\|^2_{L^2(\Om)}\right]\nonumber\\
&+ 9\|f''(u_{N})\|^2_{L^\infty(\Om)}\cd \left[\|u_{Nt}\|^2_{L^{4}(\Om)}\|\na u_{Ntt}\|^2_{L^4(\Om)}+\|u_{Ntt}\|^2_{L^4(\Om)}\|\na u_{Nt}\|^2_{L^4(\Om)}\right.\nonumber\\
&\left.\quad+\|u_{Nttt}\|^2_{L^4(\Om)}\|\na u_{N}\|^2_{L^4(\Om)}+\|u_{Nt}\|^2_{L^{4}(\Om)}\|\na u_{Nt}\|^2_{L^4(\Om)}+\|u_{Ntt}\|^2_{L^4(\Om)}\|\na u_{N}\|^2_{L^4(\Om)}\right]\nonumber\\[1mm]
&+\|f'(u_{N})\|^2_{L^\infty(\Om)}\|\na u_{Ntt}\|^2_{L^2(\Om)}
\end{align}
for all $t\in (0, \tmn)$.
\end{lem}	

\proof
For all $t\in (0, \tmn)$, an application of Young's inequality yields 
\bas
\left|\del \io \Del u_{Nt}(\cd,t)\Del u_{Ntt}(\cd,t)\right|\leq \frac{\gamma}{8}\io |\Del u_{Ntt}(\cd,t)|^2+\frac{2\del^2}{\gamma}\io |\Del u_{Nt}(\cd,t)|^2
\eas
and 
\bas
\left|\varrho \io \Del u_{N}(\cd,t)\Del u_{Ntt}(\cd,t)\right|\leq \frac{\gamma}{8}\io |\Del u_{Ntt}(\cd,t)|^2+\frac{2\varrho^2}{\gamma}\io |\Del u_{N}(\cd,t)|^2.
\eas
Inserting these estimates into (\ref{E7}), we obtain 
\bee
\spl
\F_{N}(t)&\geq \frac{1}{2}\io |\na u_{Nttt}(\cd,t)|^2+\frac{\gamma}{4}\io |\Del u_{Ntt}(\cd, t)|^2\\&\quad+\left(\mu-\frac{2\del^2}{\gamma}\right)\io |\Del u_{Nt}(\cd,t)|^2+\left(\lambda-\frac{2\varrho^2}{\gamma}\right)\io |\Del u_{N}(\cd, t)|^2
\end{split}
\eee 
from which (\ref{E8}) follows immediately, in view of (\ref{E07}). In order to prove (\ref{Def1EE}), we rely on the fact that the chosen basis $(e_{i})_{i\in \N}$ ensures $\varphi=-\Del u_{Nttt}\in C^{0}([0, \tmn); V_{N})$ for all $N\in \N$. Therefore, selecting this $\vp$ in (\ref{E5}) and integrating by parts, we deduce that
\bee{}
\spl
&\,\,\,\frac{1}{2}\frac{d}{dt}\io |\na u_{Nttt}|^2+\frac{\gamma}{2}\frac{d}{dt}\io |\Del u_{Ntt}|^2+\del \frac{d}{dt}\io \Del u_{Nt}\Del u_{Ntt}+\varrho \frac{d}{dt}\io \Del u_{N}\Del u_{Ntt}\\
=&-\alpha \io |\na u_{Nttt}|^2-\bet \io \na u_{Ntt}\cd \na u_{Nttt}+\del \io |\Del u_{Ntt}|^2+\varrho \io \Del u_{Nt}\Del u_{Ntt}\\
&+\io f''''(u_{N})u^3_{Nt}\na u_{N}\cd\na u_{Nttt}+3\io f'''(u_{N})u^2_{Nt}\na u_{Nt}\cd \na u_{Nttt}+3\io f''(u_{N})u_{Ntt}\na u_{Nt}\cd \na u_{Nttt}\\[0.5mm]
&+3\io f''(u_{N})u_{Nt}\na u_{Ntt}\cd \na u_{Nttt}+3\io f'''(u_{N})u_{Nt}u_{Ntt}\na u_{N}\cd \na u_{Nttt}+\io f''(u_{N})u_{Nttt}\na u_{N}\cd \na u_{Nttt}\\
&+\io f'(u_{N})|\na u_{Nttt}|^2+\io f'''(u_{N})u^2_{Nt}\na u_{N}\cd \na u_{Nttt}+2\io f''(u_{N})u_{Nt}\na u_{Nt}\cd \na u_{Nttt}\\[0.5mm]
&+\io f''(u_{N})u_{Ntt}\na u_{N}\cd \na u_{Nttt}+\io f'(u_{N})\na u_{Ntt}\cd \na u_{Nttt}
\end{split}
\eee
for all $t\in (0, \tmn)$. We now estimate the terms on the right-hand side. First, it is immediate that 
\bee{}
-\alpha \io |\na u_{Nttt}|^2+\io f'(u_{N})|\na u_{Nttt}|^2\leq \left[\|f'(u_{N})\|_{L^{\infty}(\Omega)}+|\alpha|\right]\io |\na u_{Nttt}|^2
\eee
for all  $t\in (0, \tmn)$, while Young's inequality implies 
\bee{}
-\bet \io \na u_{Ntt}\cd \na u_{Nttt}\leq \io |\na u_{Nttt}|^2+\bet^2\io |\na u_{Ntt}|^2,
\eee
and 
\bee{}
\varrho \io \Del u_{Nt}\Del u_{Ntt}\leq \io |\Del u_{Ntt}|^2+\varrho^2\io |\Del u_{Nt}|^2
\eee
for all $t\in (0, \tmn).$ Again Young's inequality yields 
\bee{}
\spl
\io f''''(u_{N})u^3_{Nt}\na u_{N}\cd\na u_{Nttt}&\leq \io |\na u_{Nttt}|^2+\io \big[f''''(u_{N})\big]^2u^{6}_{Nt}|\na u_{N}|^2\\
&\leq \io |\na u_{Nttt}|^2+\|f''''(u_{N})\|^2_{L^{\infty}(\Om)}\|u_{Nt}\|^6_{L^{\infty}(\Om)}\|\na u_{N}\|^2_{L^{2}(\Om)}
\end{split}
\eee
and
\bee{}
\spl
3\io f'''(u_{N})u^2_{Nt}\na u_{Nt}\cd \na u_{Nttt}&\leq \io |\na u_{Nttt}|^2+9\io \big[f'''(u_{N})\big]^2u^4_{Nt}|\na u_{Nt}|^2\\
&\leq \io |\na u_{Nttt}|^2+9\|f'''(u_{N})\|^2_{L^{\infty}(\Om)}\|u_{Nt}\|_{L^{\infty}(\Om)}^{4}\|\na u_{Nt}\|^2_{L^2(\Om)}
\end{split}
\eee 
for all $t\in (0, \tmn)$. By the same token, for all $t\in (0,\tmn)$ we infer that 
\bee{}
\spl
3\io f''(u_{N})u_{Ntt}\na u_{Nt}\cd \na u_{Nttt}&\leq \io |\na u_{Nttt}|^2+ 9 \io \big[f''(u_{N})\big]^2u^2_{Ntt}|\na u_{Nt}|^2\\
&\leq \io |\na u_{Nttt}|^2+9 \|f''(u_{N})\|^2_{L^{\infty}(\Om)}\|u_{Ntt}\|^2_{L^{4}(\Om)}\|\na u_{Nt}\|^2_{L^4(\Om)},
\end{split}
\eee 
\bee{}
3\io f''(u_{N})u_{Nt}\na u_{Ntt}\cd \na u_{Nttt}\leq \io |\na u_{Nttt}|^2+9\|f''(u_{N})\|^2_{L^{\infty}(\Omega)}\|u_{Nt}\|^2_{L^{4}(\Om)}\|\na u_{Ntt}\|^2_{L^{4}(\Om)},
\eee 
\bee{}
3\io f'''(u_{N})u_{Nt}u_{Ntt}\na u_{N}\cd \na u_{Nttt}\leq \io |\na u_{Nttt}|^2+9 \|f'''(u_{N})\|^2_{L^{\infty}(\Om)}\|u_{Nt}\|^2_{L^{\infty}(\Om)}\|u_{Ntt}\|^2_{L^{4}(\Om)}\|\na u_{N}\|^2_{L^{4}(\Om)},
\eee
\bee{}
\io f''(u_{N})u_{Nttt}\na u_{N}\cd \na u_{Nttt}\leq \io |\na u_{Nttt}|^2+\|f''(u_{N})\|^2_{L^{\infty}(\Om)}\|u_{Nttt}\|^2_{L^{4}(\Om)}\|\na u_{N}\|^2_{L^4(\Om)},
\eee 
along with 
\bee{}
\io f'''(u_{N})u^2_{Nt}\na u_{N}\cd \na u_{Nttt}\leq \io |\na u_{Nttt}|^2+\|f'''(u_{N})\|^2_{L^{\infty}(\Omega)}\|u_{Nt}\|^4_{L^{\infty}(\Omega)}\|\na u_{N}\|^2_{L^2(\Om)},
\eee 
\bee{}
2\io f''(u_{N})u_{Nt}\na u_{Nt}\cd \na u_{Nttt}\leq \io |\na u_{Nttt}|^2+4\|f''(u_{N})\|^2_{L^{\infty}(\Omega)}\|u_{Nt}\|^2_{L^{4}(\Omega)}\|\na u_{Nt}\|^2_{L^{4}(\Omega)},
\eee 
\bee{}
\io f''(u_{N})u_{Ntt}\na u_{N}\cd \na u_{Nttt}\leq \io |\na u_{Nttt}|^2+\|f''(u_{N})\|^2_{L^{\infty}(\Omega)}\|u_{Ntt}\|^2_{L^{4}(\Omega)}\|\na u_{N}\|^2_{L^{4}(\Omega)}
\eee 
and 
\bee{}
\io f'(u_{N})\na u_{Ntt}\cd \na u_{Nttt}\leq \io |\na u_{Nttt}|^2+\|f'(u_{N})\|^2_{L^{\infty}(\Omega)}\|\na u_{Ntt}\|^2_{L^2(\Omega)}.
\eee 
Furthermore, Young's inequality provides 
\bee{}
\mu\frac{d}{dt}\io |\Del u_{Nt}|^2=2\mu\io \Del u_{Ntt}\Del u_{Nt}\leq \mu\io |\Del u_{Ntt}|^2+\mu\io |\Del u_{Nt}|^2
\eee 
and
\bee{}
\lambda \frac{d}{dt}\io |\Del u_{N}|^2=2\lambda\io \Del u_{Nt}\Del u_{N}\leq \lambda\io |\Del u_{Nt}|^2+\lambda\io |\Del u_{N}|^2 
\eee 
for all $t\in (0, \tmn)$. Combining all the estimates above we arrive at (\ref{Def1EE}).
\qed	
To derive suitable a priori estimates from  (\ref{Def1EE}) on intervals of fixed  $N$-independent length, we shall invoke once more the assumption $n\leq 3$.
\begin{lem}\label{lem1.3}
For each $M>0$, there exist constants $T(M)>0$ and $C(M)>0$ such that whenever (\ref{IN}) holds with
\be{E13.1}
\|\Del u_{0}\|_{L^2(\Om)}+\|\Del u_{1}\|_{L^2(\Om)}+\|\Del u_{2}\|_{L^2(\Om)}+\|\na u_{3}\|_{L^2(\Om)}\leq M,
\ee 
the corresponding solutions $u_{N}$ of (\ref{E4})-(\ref{E5}) obey $\tmn\geq T(M)$ for all $N\in \mathbb{N}$. Moreover, for any $N\in \N$ and $t\in (0, T(M)) $, the following hold: 
\be{E13.2}
\io |\na u_{Nttt}|^2\leq C(M),
\ee 
\be{E13.3}
\io | \Del u_{Ntt}|^2\leq C(M),
\ee 
\be{E13.4}
\io | \Del u_{Nt}|^2\leq C(M),
\ee 
\be{E13.5}
\io | \Del u_{N}|^2\leq C(M),
\ee 
and
\be{E13.05}
\io | u_{Ntttt}|^2\leq C(M).
\ee 
\end{lem}
\proof	
Since $n\leq 3$, standard elliptic regularity theory and the continuity of the embeddings 
\bee
W^{2,2}  (\Om)\hra L^{\infty}(\Om), \qquad W^{1,2}(\Om)\hra L^{4}(\Om), \qquad W^{2,2}(\Om)\hra W^{1,4}(\Om),
\eee 
ensure the existence of positive constants $c_{1}, c_{2}, c_{3}$ such that 
\be{E13.55}
\|\vp\|_{L^{\infty}(\Om)}\leq c_{1}\|\Del \vp \|_{L^2(\Om)} \qquad \mbox{ for all }\;\vp \in W^{2,2}(\Om)\cap W_{0}^{1,2}(\Om),
\ee 
\be{E13.6}
\|\vp\|_{L^{4}(\Om)}\leq c_{2}\|\na\vp \|_{L^2(\Om)} \qquad \mbox{ for all }\;\vp \in W_{0}^{1,2}(\Om),
\ee 
and
\be{E13.7}
\|\na\vp\|_{L^{4}(\Om)}\leq c_{3}\|\Del \vp \|_{L^2(\Om)} \qquad \mbox{ for all }\;\vp \in W^{2,2}(\Om)\cap W_{0}^{1,2}(\Om).
\ee 
Moreover, the Poincar\'e inequality yields positive constants $c_{4},c_{5}$ satisfying
\be{E13.8}
\|\na \vp \|_{L^2(\Om)}\leq c_{4}\|\Del \vp \|_{L^2(\Om)}\qquad \mbox{ for all } \; \vp \in W^{2,2}(\Om)\cap W_{0}^{1,2}(\Om),
\ee 
and 
\be{E13.9}
\| \vp \|_{L^2(\Om)}\leq c_{5}\|\na \vp \|_{L^2(\Om)}\qquad \mbox{ for all } \; \vp \in W_{0}^{1,2}(\Om).
\ee 
Given $M>0$, and with  $\mu$ and $\lambda$ as in (\ref{E07}), we set
\be{E13.10}
c_{6}= c_{6}(M):= \frac{M^2}{2}+\frac{\gamma M^2}{2}+\del M^2+\varrho M^2+\mu M^2+\lambda M^2
\ee 
and
\be{E13.11}
c_{7}=c_{7}(M):=c_{1}\left(\frac{4(c_{6}+1)}{\lambda}\right)^{\frac{1}{2}}.
\ee 
Since $f\in C^4(\R)$, we may choose $c_{8}=c_{8}(M)>0$ such that 
\be{E13.12}
|f'(\xi)|+|f''(\xi)|+|f'''(\xi)|+|f''''(\xi)|\leq c_{8}\quad \,\mbox{ for all } \; \xi \in [-c_{7}, c_{7}].
\ee 
Subsequently,  for  $j\in \{9, \ldots, 17\}$, we define the constants  $c_{j}=c_{j}(M)$ as follows: 
\bee
c_{9}:=2(c_{8}+|\al|+11)(c_{6}+1), \qquad c_{10}:= \bet^2c^2_{4}\left(\frac{4(c_{6}+1)}{\gamma}\right), \qquad c_{11}:= (\mu+\del+1)\left(\frac{4(c_{6}+1)}{\gamma}\right),
\eee 
\bee
c_{12}:=(\mu+\lambda+\varrho^2)\left(\frac{4(c_{6}+1)}{\mu}\right), \qquad c_{13}:= 4(c_{6}+1), \qquad c_{14}:= c^{6}_{1}\left(\frac{4(c_{6}+1)}{\mu}\right)^3c^2_{8}c^2_{4}\left(\frac{4(c_{6}+1)}{\lambda}\right), 
\eee
along with 
\bee
c_{15}:=9c^2_{8}\cd \left\{c^4_{1}c^2_{4}\left(\frac{4(c_{6}+1)}{\mu}\right)^3+c^2_{1}c^2_{2}c^4_{3}\left(\frac{4^3(c_{6}+1)^3}{\gamma \mu \lambda}\right)+c^{4}_{1}c^2_{4}\left(\frac{4^3(c_{6}+1)^3}{\mu^2\lambda}\right)\right\}
\eee 
and
\bee
c_{16}:=9c^2_{8}\cd\left\{2c^2_{2}c^4_{3}\left(\frac{4^2(c_{6}+1)^2}{\mu\gamma}\right)+c^2_{2}c^2_{3}\left(\frac{8(c_{6}+1)^2}{\lambda}\right)+c^2_{2}c^4_{3}\left(\frac{4^2(c_{6}+1)^2}{\mu^2}\right)+c^2_{2}c^4_{3}\left(\frac{4^2(c_{6}+1)^2}{\lambda\gamma}\right)\right\},
\eee 
as well as
\bee
c_{17}:= c^2_{8}c^2_{4}\left(\frac{4(c_{6}+1)}{\gamma}\right).
\eee 
Then, we introduce the further constant
\be{E13.14} 
c_{18}= c_{18}(M):= \sum_{i=9}^{17}c_{i}
\ee 
and we set 
\be{E13.15}
T(M):=\frac{1}{2c_{18}}.
\ee
Assuming that (\ref{E13.1}) holds true, our goal is to show that for every $N\in \N$ the quantity $\tmn$ from Lemma \ref{lem1.1} and the functional $\F_{N}$ from (\ref{E7}) satisfy $\tmn \geq T(M) $ and 
\be{E13.16}
\F_{N}(t)\leq c_{6}+1, \qquad \mbox{for all } t\in (0, T(M)).
\ee 
To this end, we define
\be{E13.17}
S_{N}:=\left\{\wh{T}\in (0, \tmn):\; \F_{N}(t)\leq c_{6}+1\;\;\mbox{for all}\,\, t\in (0, \wh{T})\right\}.
\ee 
Since the set $S_{N}$ is nonempty, the quantity
\be{TNNN}
T_{N}:=\sup S_{N}
\ee
is well defined and belongs to $(0,\tmn]\subset (0, \infty]$. Indeed,  by the continuity of $\F_{N}$ and the Cauchy-Schwarz inequality together with (\ref{E13.10}), we obtain 
\be{E13.18}
\spl
\F_{N}(0)=\,&\frac{1}{2}\io |\na u_{3N}|^2+\frac{\gamma}{2}\io |\Del u_{2N}|^2+\del \io \Del u_{1N}\Del u_{2N}\\
&+\varrho \io \Del u_{0N}\Del u_{2N}+\mu\io |\Del u_{1N}|^2+\lambda\io |\Del u_{0N}|^2\\
\leq\,& \frac{1}{2}\io |\na u_{3}|^2+\frac{\gamma}{2}\io |\Del u_{2}|^2+\del \|\Del u_{1}\|_{L^2(\Om)}\|\Del u_{2}\|_{L^2(\Om)}+\varrho \|\Del u_{0}\|_{L^2(\Om)}\|\Del u_{2}\|_{L^2(\Om)}\\
&+\mu\io |\Del u_{1}|^2+\lambda\io |\Del u_{0}|^2\\
\leq\,& \frac{M^2}{2}+\frac{\gamma M^2}{2}+\del M^2+\varrho M^2+\mu M^2+\lambda M^2=c_{6},
\end{split}
\ee 
where we used the nonexpansive property of the projection in (\ref{E1}). On the other hand, from (\ref{E8}) we deduce that
\be{E12.19}
\begin{cases}
\displaystyle\io |\na u_{Nttt}|^2\leq 2(c_{6}+1),\medskip\\
\displaystyle\io |\Del u_{Ntt}|^2\leq \frac{4(c_{6}+1)}{\gamma},\medskip\\
\displaystyle\io |\Del u_{Nt}|^2\leq \frac{4(c_{6}+1)}{\mu},\medskip\\ 
\displaystyle\io |\Del u_{N}|^2\leq \frac{4(c_{6}+1)}{\lambda},
\end{cases}
\ee 
for all $t\in (0,T_{N})$. Combining (\ref{E13.55}),  (\ref{E13.11}) and (\ref{E12.19}), we obtain  
\be{E13.20}
\|u_{N}\|_{L^{\infty}(\Om)}\leq c_{7} \qquad \mbox{for all}\; t\in (0, T_{N}).
\ee 
Consequently, in view of (\ref{E13.12}) and (\ref{E12.19}), we derive 
\bee
\left[\|f'(u_{N})\|_{L^{\infty}(\Om)}+|\alpha|+11\right] \io |\na u_{Nttt}|^2\leq 2\{c_{8}+|\al|+11\}\cd(c_{6}+1)=c_{9}
\eee 
for all $t\in (0, T_{N})$. Moreover (\ref{E12.19}) implies 
$$
(\mu+\del+1)\io |\Del u_{Ntt}|^2\leq \frac{4(\mu+\del+1)(c_{6}+1)}{\gamma}=c_{11}$$
and
$$\bet^2\io |\na u_{Ntt}|^2\leq \bet^2c^2_{4}\left(\frac{4(c_{6}+1)}{\gamma}\right)=c_{10}
$$
for all $t\in (0,T_{N})$. Similarly
$$
(\mu+\lambda+\varrho^2)\io |\Del u_{Nt}|^2\leq (\mu+\lambda+\varrho^2) \left(\frac{4(c_{6}+1)}{\mu}\right)=c_{12}$$
and
$$\lambda\io |\Del u_{N}|^2\leq 4(c_{6}+1)=c_{13}
$$
for all $t\in (0, T_{N})$. Using (\ref{E13.55})-(\ref{E13.9}), (\ref{E12.19}) and  (\ref{E13.12}) we find
\bee 
\spl
\|f''''(u_{N})\|^2_{L^{\infty}(\Om)}\|u_{Nt}\|_{L^{\infty}(\Om)}^{6}\|\na u_{N}\|^2_{L^{2}(\Om)}&\leq c^2_{8}c^6_{1}c^2_{4}\|\Del u_{Nt}\|^6_{L^2(\Om)}\|\Del u_{N}\|^2_{L^2(\Om)}\\[0.5mm]
& \leq c^{6}_{1}\left(\frac{4(c_{6}+1)}{\mu}\right)^3c^2_{8}c^2_{4}\left(\frac{4(c_{6}+1)}{\lambda}\right)=c_{14} 
\end{split}
\eee 
and 
\bee
\spl
9\|f'''(u_{N})\|^2_{L^{\infty}(\Om)}\cd& \left[\|u_{Nt}\|^{4}_{L^{\infty}(\Om)}\|\na u_{Nt}\|^2_{L^2(\Om)}+\|u_{Nt}\|^2_{L^{\infty}(\Om)}\|u_{Ntt}\|_{L^{4}(\Om)}^2\|\na u_{N}\|^2_{L^{4}(\Om)}\right.\\
&\left.\quad+\|u_{Nt}\|^4_{L^{\infty}(\Om)}\|\na u_{N}\|^2_{L^2(\Om)}\right]\\[1mm]
&\leq 9c^2_{8}\cd \left[c^4_{1}c^2_{4}\|\Del u_{Nt}\|^6_{L^2(\Om)}+c^2_{1}c^2_{2}c^4_{3}\|\Del u_{Nt}\|^2_{L^2(\Om)}\|\Del u_{Ntt}\|^2_{L^2(\Om)}\|\Del u_{N}\|^2_{L^2(\Om)}\right.\\
&\left.\quad+c^4_{1}c^2_{4}\|\Del u_{Nt}\|^4_{L^2(\Om)}\|\Del u_{N}\|^2_{L^2(\Om)}\right]\\[1mm]
& \leq 9c^2_{8}\cd \left\{c^4_{1}c^2_{4}\left(\frac{4(c_{6}+1)}{\mu}\right)^3+c^2_{1}c^2_{2}c^4_{3}\left(\frac{4^3(c_{6}+1)^3}{\gamma \mu \lambda}\right)+c^{4}_{1}c^2_{4}\left(\frac{4^3(c_{6}+1)^3}{\mu^2\lambda}\right)\right\}=c_{15}
\end{split}
\eee 
for all $t\in (0, T_{N})$. Analogously
\bee
\spl
&9\|f''(u_{N})\|^2_{L^\infty(\Om)}\cd \left[\|u_{Nt}\|^2_{L^{4}(\Om)}\|\na u_{Ntt}\|^2_{L^4(\Om)}+\|u_{Ntt}\|^2_{L^4(\Om)}\|\na u_{Nt}\|^2_{L^4(\Om)}+\|u_{Nttt}\|^2_{L^4(\Om)}\|\na u_{N}\|^2_{L^4(\Om)}\right.\\
&\left.\quad +\|u_{Nt}\|^2_{L^{4}(\Om)}\|\na u_{Nt}\|^2_{L^4(\Om)}+\|u_{Ntt}\|^2_{L^4(\Om)}\|\na u_{N}\|^2_{L^4(\Om)}\right]\\[1mm]
& \leq 9c^2_{8}\cd \left[2c^2_{2}c^{4}_{3}\|\Del u_{Nt}\|^2_{L^{2}(\Omega)}\|\Del u_{Ntt}\|^2_{L^{2}(\Omega)}+c^2_{2}c^2_{3}\|\na u_{Nttt}\|^2_{L^2(\Om)}\|\Del u_{N}\|^2_{L^2(\Om)}\right.\\
& \left. \quad +c^2_{2}c^4_{3}\|\Del u_{Nt}\|^4_{L^2(\Om)}+c^2_{2}c^4_{3}\|\Del u_{Ntt}\|^2_{L^2(\Om)}\|\Del u_{N}\|^2_{L^2(\Om)}\right]\\[1mm]
& \leq 9c^2_{8}\cd\left\{2c^2_{2}c^4_{3}\left(\frac{4^2(c_{6}+1)^2}{\mu\gamma}\right)+c^2_{2}c^2_{3}\left(\frac{8(c_{6}+1)^2}{\lambda}\right)+c^2_{2}c^4_{3}\left(\frac{4^2(c_{6}+1)^2}{\mu^2}\right)+c^2_{2}c^4_{3}\left(\frac{4^2(c_{6}+1)^2}{\lambda\gamma}\right)\right\}\\
&=c_{16}
\end{split}
\eee 
for all $t\in (0, T_{N})$. Finally 
\bee
\|f'(u_{N})\|^2_{L^\infty(\Om)}\|\na u_{Ntt}\|^2_{L^2(\Om)}\leq c^2_{8}c^2_{4}\|\Del u_{Ntt}\|^2_{L^2(\Om)}\leq c^2_{8}c^2_{4}\left(\frac{4(c_{6}+1)}{\gamma}\right)=c_{17}, \qquad \mbox{for all } \; t\in (0, T_{N}).
\eee 
Combining the above  estimates with (\ref{Def1EE}) and recalling the definition of $c_{18}$ in (\ref{E13.14}), we arrive at 
\bee
\F'_{N}(t)\leq c_{18}, \qquad \mbox{for all }\; t\in (0, T_{N}).
\eee 
Thus, using (\ref{E13.18}), we obtain
\be{E13.30}
\F_{N}(t)\leq \F_{N}(0)+c_{18}t\leq c_{6}+c_{18}t,\qquad \mbox{for all} \; t\in (0, T_{N}).
\ee 
Now suppose that $T_{N}<T(M)$. By the continuity of $\F_{N}$, either $\tmn\leq T(M)$ or $\tmn>T(M)$ and $\F_{N}(T_{N})=c_{6}+1$. The first alternative is incompatible with (\ref{E13.30}) in light of the extensibility criterion (\ref{E6}). The second alternative is also impossible, since (\ref{E13.14}) would imply
\bee
c_{6}+1=\F_{N}(T_{N})\leq c_{6}+T_{N}c_{18}<c_{6}+c_{18}T(M)=c_{6}+\frac{1}{2},
\eee 
a contradiction. Hence $T_{N}\geq T(M)$, and (\ref{E13.16}) follows.
The estimates in (\ref{E13.2})-(\ref{E13.5}) are direct consequences of (\ref{E13.16}) combined with (\ref{E8}). Finally, to establish (\ref{E13.05}), we choose $\vp:=u_{Ntttt}$ in (\ref{E5}) and then the Cauchy-Schwarz inequality yields 
\bee
\spl
\io u^2_{Ntttt}=&-\al \io u_{Nttt}u_{Ntttt}-\bet \io u_{Ntt}u_{Ntttt}+\gamma \io \Del u_{Ntt}u_{Ntttt}+\del \io \Del u_{Nt}u_{Ntttt}+\varrho\io \Del u_{N}u_{Ntttt}\\&+\io f'(u_{N})u_{Nttt}u_{Ntttt}+\io3u_{Ntt}u_{Nt}f''(u_{N})u_{Ntttt}+\io f'''(u_{N})u^3_{Nt}u_{Ntttt}\\
& + \io f''(u_{N})u^2_{Nt}u_{Ntttt}+\io f'(u_{N})u_{Ntt}u_{Ntttt}\\[1mm]
\leq& \left[|\al|\cd \|u_{Nttt}\|_{L^2(\Om)}+\bet \|u_{Ntt}\|_{L^2(\Om)}+\gamma\|\Del u_{Ntt}\|_{L^2(\Om)}+\del \|\Del u_{Nt}\|_{L^2(\Om)}+\varrho \|\Del u_{N}\|_{L^2(\Om)}\right.
\\[1mm]
&\left. +\|f'(u_{N})\|_{L^{\infty}(\Om)}\|u_{Nttt}\|_{L^2(\Om)}
+3\|f''(u_{N})\|_{L^{\infty}(\Om)}\|u_{Ntt}\|_{L^4(\Om)}\|u_{Nt}\|_{L^4(\Om)}\right.\\[1mm]
&\left.+|\Om|^{\frac{1}{2}}\|f'''(u_{N})\|_{L^{\infty}(\Om)}\|u_{Nt}\|^3_{L^{\infty}(\Om)}+|\Om|^{\frac{1}{2}}\|f''(u_{N})\|_{L^{\infty}(\Om)}\|u_{Nt}\|^2_{L^{\infty}(\Om)}\right.\\[0.7mm]&\left.+\|f'(u_{N})\|_{L^{\infty}(\Om)}\|u_{Ntt}\|_{L^{2}(\Om)}\right]\cd \|u_{Ntttt}\|_{L^2(\Om)}
\end{split}
\eee 
for all $t\in (0,T(M))$.
Consequently, by (\ref{E13.55})-(\ref{E13.9}), and (\ref{E13.12}), we infer that
\bee
\spl
\|u_{Ntttt}\|_{L^2(\Om)}\leq&\{c_{5}|\al|+c_{8}c_{5}\}\cd\|\na u_{Nttt}\|_{L^2(\Om)}+\{c_{4}c_{5}\bet+c_{4}c_{5}c_{8}+\gamma\}\cd \|\Del u_{Ntt}\|_{L^2(\Om)}\\[1mm]
&\left.+\del  \|\Del u_{Nt}\|_{L^2(\Om)}+c_{8}c^{2}_{1}|\Om|^{\frac{1}{2}} \|\Del u_{Nt}\|^2_{L^2(\Om)}+c_{8}c^{3}_{1}|\Om|^{\frac{1}{2}}\|\Del u_{Nt}\|^3_{L^2(\Om)}+\varrho \|\Del u_{N}\|_{L^2(\Om)}\right.\\[1mm]
&+3c_{8}c^4_{2}c^4_{3}\|\Del u_{Ntt}\|^2_{L^2(\Om)}\|\Del u_{Nt}\|^2_{L^2(\Om)}
\end{split}
\eee 
for all $t\in (0, T(M))$ and $N\in \mathbb{N}.$ Therefore (\ref{E13.05}) follows directly from (\ref{E13.2})-(\ref{E13.5}).\qed
\begin{lem}\label{lem2.5}
Suppose that (\ref{IN}) holds. Then there exist $\tm\in (0, \infty] $ and at most one strong solution $u$ of (\ref{0}) in $\Om\times (0,\tm)$ satisfying
\be{E14.1}
\tm<\infty  \,\Rightarrow\, \limsup_{t\nearrow \tm}\left[\|\na u_{ttt}(\cd,t)\|_{L^2(\Om)}+\|\Del u_{tt}(\cd,t)\|_{L^2(\Om)}+\|\Del u_{t}(\cd,t)\|_{L^2(\Om)}+\|\Del u(\cd,t)\|_{L^2(\Om)}\right]=\infty.
\ee 
\end{lem}
\proof
For a fixed $M>0$ and for initial data $(u_{0}, u_{1}, u_{2}, u_{3})$ satisfying (\ref{IN}) and (\ref{E13.1}), Lemma~\ref{lem1.3} with the corresponding $T(M)$ and $C(M)$ ensures that $\tmn\geq T(M)$ for every $N\in \N$ and that (\ref{E13.2})-(\ref{E13.05}) remain valid. In particular:
\bas
(u_{N})_{N\in \N} & \mbox{ is bounded in }& L^{\infty}((0,T(M)); W^{2,2}(\Om)\cap W_{0}^{1,2}(\Om)),\\
(u_{Nt})_{N\in \N}& \mbox{ is bounded in }& L^{\infty}((0,T(M)); W^{2,2}(\Om)\cap W_{0}^{1,2}(\Om)),\\
(u_{Ntt})_{N\in \N}& \mbox{ is bounded in }& L^{\infty}((0,T(M)); W^{2,2}(\Om)\cap W_{0}^{1,2}(\Om)),\\
(u_{Nttt})_{N\in \N}& \mbox{ is bounded in }& L^{\infty}((0,T(M)); W_{0}^{1,2}(\Om)),\\
(u_{Ntttt})_{N\in \N} &\mbox{ is bounded in } & L^{\infty}((0,T(M)); L^2(\Om)).
\eas
Consequently, by virtue of the Aubin-Lions lemma \cite{Simon-1996} and the compact embeddings
\bee
W^{2,2}(\Om)\hra  W^{1,2}(\Om),\qquad  W^{2,2}(\Om)\hra C^{0}(\bom), \qquad  W^{1,2}(\Om)\hra  L^2(\Om),
\eee 
we additionally obtain
\bas
(u_{N})_{N\in \N} &\mbox{ is relatively compact in }& C^0([0,T(M)]; W_{0}^{1,2}(\Om)\cap C^{0}(\bom)),\\
(u_{Nt})_{N\in \N} & \mbox{ is relatively compact in } & C^0([0,T(M)]; W_{0}^{1,2}(\Om)\cap C^{0}(\bom)),\\
(u_{Ntt})_{N\in \N}&\mbox{ is relatively compact in }& C^0([0,T(M)]; W_{0}^{1,2}(\Om)\cap C^{0}(\bom)),\\
(u_{Nttt})_{N\in \N}& \mbox{ is relatively compact in }& C^0([0,T(M)]; L^2(\Om)).
\eas
A subsequence extraction argument yields $(N_{j})_{j\in \N}\subset \N$  with $N_{j}\to \infty$ as $j\to \infty$ and a function
\bee
u\in C^0([0,T(M)]; W_{0}^{1,2}(\Om)\cap C^{0}(\bom))\cap L^{\infty}((0,T(M)); W^{2,2}(\Om))
\eee 
fulfilling
\bee
u_{t}\in C^0([0,T(M)]; W_{0}^{1,2}(\Om)\cap C^{0}(\bom))\cap L^{\infty}((0,T(M)); W^{2,2}(\Om)),
\eee 
\bee
u_{tt}\in C^0([0,T(M)]; W_{0}^{1,2}(\Om)\cap C^{0}(\bom))\cap L^{\infty}((0,T(M)); W^{2,2}(\Om)),
\eee 
\bee
u_{ttt}\in C^0([0,T(M)]; L^2(\Om))\cap L^{\infty}((0,T(M)); W_{0}^{1,2}(\Om)),
\eee 
as well as
\bee
u_{tttt}\in  L^{\infty}((0,T(M)); L^2(\Om)).
\eee 
Moreover, when $N=N_{j}\to \infty$ as $j\to \infty$, there hold
\bea{}
u_{N}\rightarrow u &  \mbox{ in }& C^0([0,T(M)]; W_{0}^{1,2}(\Om)\cap C^{0}(\bom)),\label{1.503}\\
\Del u_{N}\wsto\Del u &  \mbox{ in }& L^{\infty}((0,T(M)); L^2(\Om)),\label{1.53}\\
u_{Nt}\rightarrow u_{t} &  \mbox{ in }& C^0([0,T(M)]; W_{0}^{1,2}(\Om)\cap C^{0}(\bom))\label{1.54},\label{E1.55}\\
\Del u_{Nt}\wsto \Del u_{t} &  \mbox{ in }& L^{\infty}((0,T(M)); L^2(\Om)),\label{1.55}\\
u_{Ntt}\rightarrow u_{tt} &  \mbox{ in }& C^0([0,T(M)]; W_{0}^{1,2}(\Om)\cap C^{0}(\bom)),\label{1.56}\\
\Del u_{Ntt}\wsto \Del u_{tt} &  \mbox{ in }& L^{\infty}((0,T(M)); L^2(\Om)),\label{1.57}\\
u_{Nttt}\rightarrow u_{ttt} &  \mbox{ in }& C^{0}([0,T(M)]; L^2(\Om)),\label{1.58}\\
u_{Ntttt}\wsto u_{tttt} &  \mbox{ in }& L^{\infty}((0,T(M)); L^2(\Om)).\label{1.59}\\ \nn 
\eea 
To confirm that $u$ possesses the claimed strong solution properties, observe that (\ref{E1}), together with  (\ref{1.503}), (\ref{E1.55}), (\ref{1.56}) and (\ref{1.58}) imply  $u(\cd,0)=\lim_{j\to \infty} u_{N_{j}}(\cd,0)=\lim_{j\to \infty} u_{0N_{j}}=u_{0}$ a.e.\ in $\Om$. Analogously, we obtain $u_{t}(\cd, 0)=u_{1}$, $u_{tt}(\cd,0)=u_{2}$ and $u_{ttt}(\cd,0)=u_{3}$ a.e.\ in $\Om$. Moreover, for any $N_{0}\in \N$ and $\vp \in L^2((0,T(M)); V_{N_{0}})$, it follows from (\ref{E5}) that 
\be{E1.602}
\spl
&\int_{0}^{T(M)}\io u_{Ntttt}\vp+\al \int_{0}^{T(M)}\io u_{Nttt}\vp+\bet \int_{0}^{T(M)}\io u_{Ntt}\vp-\gamma \int_{0}^{T(M)}\io \Del u_{Ntt}\vp\\[1mm]
=\,&\del\int_{0}^{T(M)}\io \Del u_{Nt}\vp+\varrho \int_{0}^{T(M)}\io \Del u_{N}\vp+\int_{0}^{T(M)}\io u_{Nttt}f'(u_{N})\vp\\[1mm]
&+3\int_{0}^{T(M)}\io u_{Ntt}u_{Nt}f''(u_{N})\vp+\int_{0}^{T(M)}\io u^3_{Nt}f'''(u_{N})\vp\\[1mm]&+\int_{0}^{T(M)}\io u^2_{Nt}f''(u_{N})\vp+\int_{0}^{T(M)}\io u_{Ntt}f'(u_{N})\vp
\end{split}
\ee 
for every $N\geq N_{0}$. Since (\ref{1.503}) and the continuity of $f'$, $f''$, $f'''$ guarantee that
\bee
f'(u_{N})  \rightarrow f'(u),\quad\,\, f''(u_{N})\rightarrow f''(u) \quad \mbox{ and }\quad f'''(u_{N})\rightarrow f'''(u)\;\,\,\mbox{ in }\;C^0(\bom\times [0,T(M)])
\eee 
as $N=N_{j}\to \infty$. Moreover (\ref{E1.55}), (\ref{1.56}) and (\ref{1.58}) imply that
\bee
u_{Nttt}f'(u_{N})\rightarrow u_{ttt}f'(u)\quad\,\, \mbox{ and }\quad\,\, u_{Ntt}u_{Nt}f''(u_{N})\rightarrow u_{tt}u_{t}f''(u)\quad \mbox{ in }\; L^1(\Om\times (0,T(M))),
\eee  
\bee
u^3_{Nt}f'''(u_{N})\rightarrow u^3_{t}f'''(u)\quad\,\, \mbox{ and }\quad\,\, u^2_{Nt}f''(u_{N})\rightarrow u^2_{t}f''(u)\quad \mbox{ in }\; L^1(\Om\times (0,T(M))),
\eee  
as well as 
\bee
u_{Ntt}f'(u_{N})\rightarrow u_{tt}f'(u)\quad \mbox{ in }\;\; L^1(\Om\times (0,T(M)))
\eee  
as $N=N_{j}\to \infty$. Combining these convergences with (\ref{E1.602}) and (\ref{1.53})-(\ref{1.59}) we find 
\bee{}
\spl
&\int_{0}^{T(M)}\io u_{tttt}\vp+\al \int_{0}^{T(M)}\io u_{ttt}\vp+\bet \int_{0}^{T(M)}\io u_{tt}\vp+\gamma \int_{0}^{T(M)}\io \Del u_{tt}\vp\\[1mm]
=\,&\,\del\int_{0}^{T(M)}\io \Del u_{t}\vp+\varrho \int_{0}^{T(M)}\io \Del u\vp+\int_{0}^{T(M)}\io u_{ttt}f'(u)\vp\\[1mm]
&+3\int_{0}^{T(M)}\io u_{tt}u_{t}f''(u)\vp+\int_{0}^{T(M)}\io u^3_{t}f'''(u)\vp+\int_{0}^{T(M)}\io u^2_{t}f''(u)\vp+\int_{0}^{T(M)}\io u_{tt}f'(u)\vp.
\end{split}
\eee 
By virtue of the density of $\bigcup_{N_{0}\in \N} V_{N_{0}}$  in $L^2(\Om)$, this identity holds for every $\vp \in L^2(\Om \times (0,T(M)))$. Hence (\ref{Def1E}) is satisfied. Uniqueness follows from Lemma \ref{lem2.4}, while the extensibility criterion (\ref{E14.1}) follows from  a standard continuation argument, because the chosen  $T(M)$ depends on the initial data $(u_{0}, u_{1}, u_{2}, u_{3})$ solely through the quantities featured in (\ref{E13.1}). 
\qed
%
%
\section{Ruling out the gradient blow-up:\ Proof of Theorem \ref{Th1}}\label{RO}
In this section, we turn to the main argument underlying the proof of Theorem \ref{Th1}. Due to Lemma~\ref{lem2.5}, our objective is to establish the sharpened extensibility criterion  (\ref{BU}). To this end, we derive suitable bounds for the quantities  
$$  \io |\na u_{ttt} |^2,\qquad \io |\Del u_{tt}|^2,\qquad \io |\Del u_{t}|^2\qquad\mbox{ and }\qquad \io |\Del u|^2 $$ 
whenever $\tm<\infty$ and 
$$ \sup_{t\in (0,\tm)}\|u(\cd,t)\|_{L^{\infty}(\Om)}<\infty.$$ 

The key step in the proof of Lemma \ref{lem1.9} is to test the lifted equation  (\ref{E15.07}) in a carefully chosen manner:\ this allows us to derive the required estimate while overcoming the potential difficulty caused by the lack of regularity. We emphasize that the full set of regularity assumptions in  (\ref{IN0}) will be used throughout this section. 
\begin{lem}\label{lem1.9}
In addition to (\ref{IN}) suppose that 
\be{IN2}
u_{0},u_1\in W^{4,2}(\Om)\qquad  \mbox{and}\qquad u_{3}\in W^{2,2}(\Om).
\ee
Assume that the strong solution of (\ref{0}) exists on $(0,\tm)$ with  $\tm<\infty$ and that, for some $K>0$,
\be{E16.1}
\|u(\cd,t)\|_{L^{\infty}(\Om)}\leq K \qquad \mbox{for all }\;t\in (0, \tm).
\ee 
Then there exists $C_{1}>0$ such that 
\be{E16.2}
\io |\Del u(\cd,t)|^2\leq C_{1} \qquad \mbox{for a.e. } t\in (0,\tm).
\ee 
\end{lem}
\proof
Let $A$ be defined as in Lemma \ref{lem1.8}. For $\eps>0$, we consider 
\be{E16.3}
\spl
\F^{(1)}_{\eps}(t)&:=\frac{1}{2}\io|(1+\eps A)^{-1}\Del u(\cd,t)|^2+\frac{\gamma}{2}\io |\na \left\{(1+\eps A)^{-1}\Del v(\cd,t)\right\}|^2\\
&\quad+\del \io \na \left\{(1+\eps A)^{-1}\Del v(\cd,t)\right\}\cd \na \left\{(1+\eps A)^{-1}\Del w(\cd,t)\right\}\\&\quad+
\varrho \io \na \left\{(1+\eps A)^{-1}\Del v(\cd,t)\right\}\cd \na \left\{(1+\eps A)^{-1}\Del z(\cd,t)\right\}\\&\quad+
\mu \io |\na \left\{(1+\eps A)^{-1}\Del w(\cd,t)\right\}|^2
+\lam \io |\na \left\{(1+\eps A)^{-1}\Del z(\cd,t)\right\}|^2
\end{split}
\ee 
for $t\in [0,\tm)$, where $v,w,z$ are as in (\ref{E15.1}) while $\mu$ and $\lam$ are as in (\ref{E07}). Observing that
\bee
(1+\eps A)^{-1}\Del u =-\frac{1}{\eps}u +\frac{1}{\eps}(1+\eps A)^{-1}u,
\eee 
and since $(1+\eps A)^{-1}$ acts as a bounded operator on $L^2(\Om)$, the regularity $u\in C^0([0,\tm); L^2(\Om))$ and $u_{t}\in  C^0([0,\tm); L^2(\Om))$ in (\ref{Def1}) guarantees that 
\bee
[0, \tm)\ni t\mapsto \frac{1}{2}\io |(1+\eps A)^{-1}\Del u(\cd,t)|^2
\eee 
is continuously differentiable with 
\bee
\spl
\frac{1}{2}\frac{d}{dt}\io  |(1+\eps A)^{-1}\Del u(\cd,t)|^2&=\io (1+\eps A)^{-1}\Del u\cd (1+\eps A)^{-1}\Del u_{t}\\
&= \io (1+\eps A)^{-1}\Del u\cd \Del \left\{(1+\eps A)^{-1}u_{t}\right\}
\end{split}
\eee 
due to (\ref{P1}). Moreover
\be{GD}
\na \left\{(1+\eps A)^{-1}\Del v\right\}=-\frac{1}{\eps}\int_{0}^{t}\na u+\frac{1}{\eps}\na(1+\eps A)^{-1}\int_{0}^{t} u\in C^{1}([0,\tm); L^2(\Om; \R^n)).
\ee
Indeed, by (\ref{Def1}),  we have $\na \left\{(1+\eps A)^{-1}\Del v\right\}\in W_{loc}^{2, \infty}([0,\tm);L^2(\Om;\R^n))$ and, since 
$$ W_{loc}^{2, \infty}([0,\tm);L^2(\Om;\R^n))\hookrightarrow C^1([0,\tm); L^2(\Om;\R^n)),$$ the claimed regularity in (\ref{GD}) follows. The same argument applies to  $\na\left\{(1+\eps A)^{-1}\Del w\right\}$ and  $\na\left\{(1+\eps A)^{-1}\Del z\right\}$, so that all the maps appearing in $\F^{(1)}_{\eps}$ belong to $C^1([0, \tm))$. Thus, integrating by parts, we infer that for all $t\in (0,\tm)$ 
\bee 
\spl
\frac{\gamma}{2}\frac{d}{dt}\io|\na \{(1+\eps A)^{-1}\Del v\}|^2&=\gamma \io \na \left\{(1+\eps A)^{-1}\Del v\right\}\cd \na \left\{(1+\eps A)^{-1}\Del u\right\}\\
&=-\gamma\io (1+\eps A)^{-1}\Del u \cd \Del \left\{(1+\eps A)^{-1}\Del v\right\}
\end{split}
\eee 
and
\bee
\spl
&\del\,\frac{d}{dt}\io \na \left\{(1+\eps A)^{-1}\Del v\right\}\cd\na \left\{(1+\eps A)^{-1}\Del w\right\}\\
=\,& \del \io \na \left\{(1+\eps A)^{-1}\Del u\right\}\cd\na \left\{(1+\eps A)^{-1}\Del w\right\}+\del \io |\na \left\{(1+\eps A)^{-1}\Del v\right\}|^2\\
=& -\del \io (1+\eps A)^{-1}\Del u\cd\Del \left\{(1+\eps A)^{-1}\Del w\right\}+\del \io |\na \left\{(1+\eps A)^{-1}\Del v\right\}|^2,
\end{split}
\eee 
as well as 
\bee
\spl
&\,\varrho\frac{d}{dt}\io \na \left\{(1+\eps A)^{-1}\Del v\right\}\cd\na \left\{(1+\eps A)^{-1}\Del z\right\}\\
=\,& \varrho \io \na \left\{(1+\eps A)^{-1}\Del u\right\}\cd\na \left\{(1+\eps A)^{-1}\Del z\right\}+\varrho \io \na \left\{(1+\eps A)^{-1}\Del v\right\}\cd\na  \left\{(1+\eps A)^{-1}\Del w\right\}\\
=& -\varrho \io (1+\eps A)^{-1}\Del u\cd\Del \left\{(1+\eps A)^{-1}\Del z\right\}+\varrho \io \na \left\{(1+\eps A)^{-1}\Del v\right\}\cd\na  \left\{(1+\eps A)^{-1}\Del w\right\}.
\end{split}
\eee
By the same token 
\bee
\mu\, \frac{d}{dt}\io |\na \left\{(1+\eps A)^{-1}\Del w\right\}|^2=2\mu \io \na \left\{(1+\eps A)^{-1}\Del v\right\}\cd  \na \left\{(1+\eps A)^{-1}\Del w\right\}
\eee 
and 
\bee
\lambda \,\frac{d}{dt}\io |\na \left\{(1+\eps A)^{-1}\Del z\right\}|^2=2\lambda \io \na \left\{(1+\eps A)^{-1}\Del w\right\}\cd  \na \left\{(1+\eps A)^{-1}\Del z\right\}.
\eee  
Combining the identities above with (\ref{E15.07}), we obtain 
\be{E16.4}
\spl
\frac{d}{dt}\F^{(1)}_{\eps}(t)=&\io (1+\eps A)^{-1}\Del u\cd \Del\left\{(1+\eps A)^{-1}\left\{ u_{t}-\gamma \Del v-\del \Del w-\varrho\Del z\right\}\right\}\\
&+ \del \io |\na \left\{(1+\eps A)^{-1}\Del v\right\}|^2+ \varrho \io \na \left\{(1+\eps A)^{-1}\Del v\right\}\cd\na  \left\{(1+\eps A)^{-1}\Del w\right\}\\
&+2\mu \io \na \left\{(1+\eps A)^{-1}\Del v\right\}\cd  \na \left\{(1+\eps A)^{-1}\Del w\right\}\\
&+ 2\lambda \io \na \left\{(1+\eps A)^{-1}\Del w\right\}\cd  \na \left\{(1+\eps A)^{-1}\Del z\right\}\\
=&-\al \io |(1+\eps A)^{-1}\Del u|^2-\frac{\bet}{2} \frac{d}{dt}\io |(1+\eps A)^{-1}\Del v|^2\\
&+\io (1+\eps A)^{-1}\Del u\cd \Del (1+\eps A)^{-1}\left\{f(u)+F(u)+\frac{t^2}{2}\psi_{1}+t\psi_{2}+\psi_{3}\right\}\\
&+ \del \io |\na \left\{(1+\eps A)^{-1}\Del v\right\}|^2+ \varrho \io \na \left\{(1+\eps A)^{-1}\Del v\right\}\cd\na  \left\{(1+\eps A)^{-1}\Del w\right\}\\
&+2\mu \io \na \left\{(1+\eps A)^{-1}\Del v\right\}\cd  \na \left\{(1+\eps A)^{-1}\Del w\right\}\\
&+ 2\lambda \io \na \left\{(1+\eps A)^{-1}\Del w\right\}\cd  \na \left\{(1+\eps A)^{-1}\Del z\right\}
\end{split}
\ee 
for all $t\in (0,\tm)$.
The choice of $\mu$ and $\lambda$ in Lemma \ref{lem1.2} implies, via  Young's inequality, that
\be{E16.5}
\spl
\F^{(1)}_{\eps}(t)\geq& \,\frac{1}{2}\io |(1+\eps A)^{-1}\Del u|^2+\frac{\gamma}{4}\io |\na \left\{(1+\eps A)^{-1}\Del v\right\}|^2+\frac{\mu}{4}\io |\na \left\{(1+\eps A)^{-1}\Del w\right\}|^2\\[1mm]
&+ \frac{\lambda}{4}\io |\na \left\{(1+\eps A)^{-1}\Del z\right\}|^2
\end{split} 
\ee 
for all $t\in (0, \tm)$. Moreover, using Young's inequality once more, we infer that for all $t\in (0,\tm)$
\be{E16.6}
\spl
&-\al \io |(1+\eps A)^{-1}\Del u|^2+ \del \io |\na \left\{(1+\eps A)^{-1}\Del v\right\}|^2\\
&+ \varrho \io \na \left\{(1+\eps A)^{-1}\Del v\right\}\cd\na  \left\{(1+\eps A)^{-1}\Del w\right\}\\
&+2\mu \io \na \left\{(1+\eps A)^{-1}\Del v\right\}\cd  \na \left\{(1+\eps A)^{-1}\Del w\right\}\\
&+ 2\lambda \io \na \left\{(1+\eps A)^{-1}\Del w\right\}\cd  \na \left\{(1+\eps A)^{-1}\Del z\right\}\\
\leq &\,|\alpha|\io |(1+\eps A)^{-1}\Del u|^2+(\del +\mu+\varrho)\io |\na \left\{(1+\eps A)^{-1}\Del v\right\}|^2\\
&+(\varrho+\mu+\lambda)\io |\na \left\{(1+\eps A)^{-1}\Del w\right\}|^2
+\lambda\io |\na \left\{(1+\eps A)^{-1}\Del z\right\}|^2\\[1.5mm]
\leq &\, c_{1}\F^{(1)}_{\eps}(t),
\end{split}
\ee
where $c_{1}:=2|\al|+\frac{4}{\gamma}(\del +\mu+\varrho)+\frac{4}{\mu}(\varrho+\mu+\lambda)+4.$ Furthermore, due to
(\ref{P1})-(\ref{P2}) and the Young's inequality,  
\be{E16.7}
\spl
&\io (1+\eps A)^{-1}\Del u\cd \Del (1+\eps A)^{-1}\left\{\frac{t^2}{2}\psi_{1}+t\psi_{2}+\psi_{3}\right\}\\[1mm]
=& \frac{t^2}{2}\io (1+\eps A)^{-1}\Del u\cd (1+\eps A)^{-1}\Del \psi_{1}+t \io (1+\eps A)^{-1}\Del u\cd (1+\eps A)^{-1}\Del \psi_{2} \\
&+\io (1+\eps A)^{-1}\Del u\cd (1+\eps A)^{-1}\Del\psi_{3}\\
\leq &\,3\io |(1+\eps A)^{-1}\Del u|^2+\frac{t^4}{4}\io |(1+\eps A)^{-1}\Del \psi_{1} |^2\\
&+t^2\io |(1+\eps A)^{-1}\Del \psi_{2} |^2+\io |(1+\eps A)^{-1}\Del \psi_{3} |^2\\[1mm]
\leq &\,6\,\F^{(1)}_{\eps }(t)+\frac{t^4}{4}\io |\Del \psi_{1} |^2+t^2\io |\Del \psi_{2} |^2+\io |\Del \psi_{3} |^2\\[1mm]
\leq &\, 6\,\F^{(1)}_{\eps }(t)+c_{2}
\end{split} 
\ee 
for all $t\in (0, \tm)$ where, thanks to (\ref{IN2}), $f\in C^4(\R)$ and the assumption $\tm <\infty$
\bee
c_{2}:=\frac{\tm^4}{4}\io |\Del \psi_{1} |^2+\tm^2\io |\Del \psi_{2} |^2+\io |\Del \psi_{3} |^2<\infty.
\eee 
To  estimate the  terms involving $f(u)$ and $F(u)$ in (\ref{E16.4}), we  proceed similarly:
\be{E16.20}
\spl
&\quad\io (1+\eps A)^{-1}\Del u\cd \Del \left\{(1+\eps A)^{-1}f(u)\right\}+\io (1+\eps A)^{-1}\Del u\cd \Del \left\{(1+\eps A)^{-1}F(u)\right\}\\[1mm]
&=\io (1+\eps A)^{-1}\Del u\cd \left\{(1+\eps A)^{-1}\Del f(u)\right\}+\io (1+\eps A)^{-1}\Del u\cd \left\{(1+\eps A)^{-1}\Del F(u)\right\}\\[1mm]
&\leq \, 2\io |(1+\eps A)^{-1}\Del u|^2+\io |\Del f(u)|^2+\io |\Del F(u)|^2\\
&\leq 4\,\F^{(1)}_{\eps}(t)+ \io |f'(u) \Del u+|\na u|^2f''(u)|^2+\io \left|\int_{0}^{t}f'(u) \Del u+|\na u|^2f''(u)\right|^2\\
&\leq 4\,\F^{(1)}_{\eps}(t)+2\io [f'(u)]^2|\Del u|^2 +2\io [f''(u)]^2|\na u|^4+2t\int_{0}^{t}\io [f'(u)]^2|\Del u|^2+2t \int_{0}^{t}\io [f''(u)]^2|\na u|^4
\end{split} 
\ee
for all $t\in (0,\tm) $. At this point we invoke the Gagliardo-Nirenberg inequality which, together with elliptic regularity theory, gives $c_{3} >0$ such that 
\be{GN}
\io |\na \vp |^4\leq c_{3}\|\Del \vp \|^2_{L^2(\Om)}\| \vp \|^2_{L^\infty(\Om)} \qquad\, \mbox{for all}\;\;\vp \in W^{2,2}(\Om)\cap W_{0}^{1,2}(\Om).
\ee
By (\ref{E16.1}) and the regularity of $f$ in (\ref{f}), there exist $c_{4},c_{5}>0$ such that 
$$[f'(u)]^2\leq c_{4}\qquad \text{and}\qquad [f''(u)]^2\leq c_{5} \quad\,\, \text{ in }\,\, \Om \times (0, \tm),$$ 
whence
\bee
\spl
2\io [f'(u)]^2|\Del u|^2 +2\io [f''(u)]^2|\na u|^4
\leq& c_{4}\io |\Del u|^2+2c_{5}\io |\na u|^4\\
\leq &c_{4}\io |\Del u|^2+2c_{5}c_{3}\|u\|^2_{L^{\infty}(\Om)}\io |\Del u|^2\\
\leq& c_{6}\io |\Del u|^2
\end{split}
\eee 
having set $c_{6}:=c_{4}+2c_{5}c_{3}K^2.$ Similarly,
\bee
\spl
2t\int_{0}^{t}\io [f'(u)]^2|\Del u|^2 +2t\int_{0}^{t}\io [f''(u)]^2|\na u|^4
\leq& c_{4}\tm \int_{0}^{t}\io |\Del u|^2+2c_{5}\tm\int_{0}^{t}\io |\na u|^4\\
\leq &c_{4}\tm \int_{0}^{t}\io |\Del u|^2+2c_{5}c_{3}\tm\int_{0}^{t}\|u\|^2_{L^{\infty}(\Om)}\io |\Del u|^2\\
\leq& c_{7}\int_{0}^{t}\io |\Del u|^2
\end{split}
\eee 
with $c_{7}:= \tm c_{4}+2\tm c_{5}c_{3}K^2$. Combining (\ref{E16.4}) and (\ref{E16.6})-(\ref{E16.20}) we have
\bee
\frac{d}{dt}\F^{(1)}_{\eps}(t)+\frac{\bet}{2} \frac{d}{dt}\io |(1+\eps A)^{-1}\Del v|^2\leq (c_{1}+10)\F^{(1)}_{\eps}(t)+c_{6}\io |\Del u|^2+c_{7}\int_{0}^{t}\io |\Del u|^2+c_{2}
\eee
for all $t\in (0, \tm)$.
After an integration with respect to time, we arrive at
\be{E16.20r}
\spl
&\quad\,\,\F^{(1)}_{\eps}(t)+\frac{\bet e^{-(c_{1}+10)t}}{2}\io |(1+\eps A)^{-1}\Del v|^2+\frac{\bet(c_{1}+10)}{2}\int_{0}^{t}\io e^{(c_{1}+10)(t-s)} |(1+\eps A)^{-1}\Del v|^2\\[1mm]
&\leq \F^{(1)}_{\eps}(0)+c_{6}\int_{0}^{t}e^{(c_{1}+10)(t-s)}\io |\Del u|^2\,ds \\
&\quad+ c_{7}\int_{0}^{t}e^{(c_{1}+10)(t-s)}\int_{0}^{s}\io |\Del u|^2\,d\sigma ds+c_{2}\int_{0}^{t}e^{(c_{1}+10)(t-s)}\,ds\\[1mm]
&\leq c_{8}+c_{9}\int_{0}^{t}\io |\Del u|^2
\end{split}
\ee
where $c_{8}:=\left\{\frac{1}{2}\io |\Del u_{0}|^2+\frac{c_{2}}{c_{1}+10}\right\}\cd e^{(c_{1}+10)\tm}$ and $c_{9}:=\left\{c_{6}+c_{7}\tm \right\}\cd e^{(c_{1}+10)\tm}$. Here we used $v(\cd, 0)\equiv 0$, $w(\cd,0)\equiv 0$ and $z(\cd,0)\equiv 0$, which imply 
$$ \F^{(1)}_{\eps}(0)=\frac{1}{2}\io |(1+\eps A)^{-1}\Del u_{0}|^2\leq \frac{1}{2}\io |\Del u_{0}|^2$$
by (\ref{P2}). Moreover, for all $t\in (0, \tm)$
$$ \int_{0}^{t}e^{(c_{1}+10)(t-s)}\,ds\leq \frac{e^{(c_{1}+10)\tm}}{c_{1}+10}\left(1-e^{-(c_{1}+10)t}\right)$$
and
\bee
\spl
\int_{0}^{t}e^{(c_{1}+10)(t-s)}\int_{0}^{s}\io |\Del u|^2\,d\sigma ds&\leq e^{(c_{1}+10)\tm} \int_{0}^{t}\int_{0}^{s}\io |\Del u|^2\,d\sigma ds\\[1mm]
&=e^{(c_{1}+10)\tm}\int_{0}^{t}(t-\sigma)\io |\Del u|^2\,d\sigma\\[1mm]
&\leq \tm e^{(c_{1}+10)\tm}\int_{0}^{t}\io |\Del u|^2.
\end{split}
\eee 
In view of (\ref{E16.5}), dropping the nonnegative terms on the left-hand side of (\ref{E16.20r}) we obtain 
$$ \frac{1}{2}\io |(1+\eps A)^{-1}\Del u(\cd,t)|^2\leq c_{8}+c_{9} \int_{0}^{t}\io |\Del u|^2 \qquad \mbox{for all}\; t\in (0,\tm).$$
According to (\ref{P4}), and since $\Del u \in L_{loc}^{\infty}([0,\tm); L^{2}(\Om))$, as $\eps \searrow 0$ the function $y(t):=\io |\Del u(\cd,t)|^2$ belongs to $L_{loc}^{\infty}([0, \tm))$ and satisfies  
$$ y(t)\leq 2c_{8}+2c_{9}\int_{0}^{t}y(s)ds \qquad \mbox{for a.e. } t\in (0,\tm).$$
An application of Gronwall's inequality now yields 
$$ y(t)\leq 2c_{8}e^{2c_{9}t}\qquad \mbox{for a.e. } t\in (0,\tm).$$
In view of the finiteness of $\tm$, inequality (\ref{E16.2}) follows.
\qed  
The argument in the proof of Lemma \ref{lem1.9} can be applied to establish also the next lemma. The only difference is that now we test the lifted equation (\ref{E15.6}) once more in a carefully chosen manner. This yields the required  estimate despite the lack of regularity.

\begin{lem}\label{lem1.10}
Under the same assumptions  of Lemma \ref{lem1.9}, there exists 
$C_{2}>0$ such that
\be{E16.22}
\io |\Del u_{t}(\cd,t)|^2\leq C_{2} \qquad \mbox{for a.e. } t\in (0,\tm).
\ee 
\end{lem}
\proof
For $\eps>0$, we consider the  functional  
\be{EE10.1}
\spl
\F^{(2)}_{\eps}(t):=&\,\frac{1}{2}\io |(1+\eps A)^{-1}\Del u_{t}|^2+\frac{\gamma}{2}\io |\na \left\{(1+\eps A)^{-1}\Del u\right\}|^2\\
&+\del \io \na \left\{(1+\eps A)^{-1}\Del v\right\}\cd \na \{(1+\eps A)^{-1}\Del u\}\\&+\varrho\io \na\{ (1+\eps A)^{-1}\Del w\}\cd \na \{(1+\eps A)^{-1}\Del u\}+\mu \io |\na \left\{(1+\eps A)^{-1}\Del v\right\}|^2\\&+\lambda \io |\na \left\{(1+\eps A)^{-1}\Del w\right\}|^2, \qquad t\in (0,\tm).
\end{split}
\ee 
In view of (\ref{Def1}) and analogously to the proof of Lemma \ref{lem1.9}, we have $\F^{(2)}_{\eps}\in C^1([0,\tm))$. Furthermore, integrating by parts and using (\ref{E15.6}) and (\ref{P1}), the time derivative of $\F^{(2)}_{\eps}$ satisfies  
\begin{align}\label{EE10.2}
&\,\frac{d}{dt}\F^{(2)}_{\eps}(t)\nonumber\\[1mm]
={}&\io (1+\eps A)^{-1}\Del u_{t}\cd
(1+\eps A)^{-1}\Del u_{tt}
+\gamma \io \na \{(1+\eps A)^{-1}\Del u\}\cd
\na \{(1+\eps A)^{-1}\Del u_{t}\}
\nonumber\\
&+\del \io \left|\na \left\{(1+\eps A)^{-1}\Del u\right\}\right|^2
+\del \io \na \left\{(1+\eps A)^{-1}\Del v\right\}\cd
\na \left\{(1+\eps A)^{-1}\Del u_{t}\right\}
\nonumber\\
&+\varrho \io \na \left\{(1+\eps A)^{-1}\Del v\right\}\cd
\na \left\{(1+\eps A)^{-1}\Del u\right\}
+\varrho \io \na \left\{(1+\eps A)^{-1}\Del w\right\}\cd
\na \left\{(1+\eps A)^{-1}\Del u_{t}\right\}
\nonumber\\
&+2\mu \io \na \left\{(1+\eps A)^{-1}\Del u\right\}\cd
\na \left\{(1+\eps A)^{-1}\Del v\right\}
+2\lambda\io \na \left\{(1+\eps A)^{-1}\Del v\right\}\cd
\na \left\{(1+\eps A)^{-1}\Del w\right\}
\nonumber\\
={}&\io (1+\eps A)^{-1}\Del u_{t}\cd
\Del (1+\eps A)^{-1}
\left\{u_{tt}-\gamma \Del u-\del \Del v-\varrho \Del w\right\}
\nonumber\\
&+\varrho \io \na \left\{(1+\eps A)^{-1}\Del v\right\}\cd
\na \left\{(1+\eps A)^{-1}\Del u\right\}
+2\mu \io \na \left\{(1+\eps A)^{-1}\Del u\right\}\cd
\na \left\{(1+\eps A)^{-1}\Del v\right\}
\nonumber\\
&+2\lambda \io \na \left\{(1+\eps A)^{-1}\Del v\right\}\cd
\na \left\{(1+\eps A)^{-1}\Del w\right\}
\nonumber\\
={}&\io (1+\eps A)^{-1}\Del u_{t}\cd
\Del (1+\eps A)^{-1}
\left\{-\al u_{t}-\bet u+f'(u)u_{t}+f(u)+t\psi_{1}+\psi_{2}\right\}
\nonumber\\
&+\varrho \io \na \left\{(1+\eps A)^{-1}\Del v\right\}\cd
\na \left\{(1+\eps A)^{-1}\Del u\right\}
+2\mu \io \na \left\{(1+\eps A)^{-1}\Del u\right\}\cd
\na \left\{(1+\eps A)^{-1}\Del v\right\}
\nonumber\\
&+2\lambda \io \na \left\{(1+\eps A)^{-1}\Del v\right\}\cd
\na \left\{(1+\eps A)^{-1}\Del w\right\}
\nonumber\\
={}&-\al \io |(1+\eps A)^{-1}\Del u_{t}|^2
-\frac{\bet}{2}\frac{d}{dt}
\io |(1+\eps A)^{-1}\Del u|^2
\nonumber\\
&+\io (1+\eps A)^{-1}\Del u_{t}\cd
\Del (1+\eps A)^{-1}
\left\{f'(u)u_{t}+f(u)+t\psi_{1}+\psi_{2}\right\}
\nonumber\\
&+\varrho \io \na \left\{(1+\eps A)^{-1}\Del v\right\}\cd
\na \left\{(1+\eps A)^{-1}\Del u\right\}
+2\mu \io \na \left\{(1+\eps A)^{-1}\Del u\right\}\cd
\na \left\{(1+\eps A)^{-1}\Del v\right\}
\nonumber\\
&+2\lambda\io \na \left\{(1+\eps A)^{-1}\Del v\right\}\cd
\na \left\{(1+\eps A)^{-1}\Del w\right\}
\end{align}
for all $t\in(0,\tm)$. The choice of $\mu $ and $\lambda $, along with  Young's inequality, yields 
\be{EE10.3}
\spl
\F^{(2)}_{\eps}(t)\geq & \,\frac{1}{2}\io |(1+\eps A)^{-1}\Del u_{t}|^2+\frac{\gamma}{4}\io |\na \left\{(1+\eps A)^{-1}\Del u\right\}|^2+\frac{\mu}{4}\io |\na \left\{(1+\eps A)^{-1}\Del v\right\}|^2\\&+\frac{\lambda}{4}\io |\na \left\{(1+\eps A)^{-1}\Del w\right\}|^2
\end{split}
\ee
for all $t\in (0,\tm)$. Another application  of Young's inequality gives 
\bee{}
\spl
&-\al \io |(1+\eps A)^{-1}\Del u_{t}|^2+\varrho \io \na \left\{(1+\eps A)^{-1}\Del v\right\}\cd \na \left\{(1+\eps A)^{-1}\Del u\right\}\\&+2\mu \io \na \left\{(1+\eps A)^{-1}\Del u\right\}\cd \na \left\{(1+\eps A)^{-1}\Del v\right\}+2\lambda\io \na \left\{(1+\eps A)^{-1}\Del v\right\}\cd \na \left\{(1+\eps A)^{-1}\Del w\right\}\\
\leq& |\al |\io |(1+\eps A)^{-1}\Del u_{t}|^2+(\varrho+\mu+\lambda) \io |\na \left\{(1+\eps A)^{-1}\Del v\right\}|^2+(\varrho+\mu) \io |\na \left\{(1+\eps A)^{-1}\Del u\right\}|^2\\
&+\lambda\io |\na \left\{(1+\eps A)^{-1}\Del w\right\}|^2\\[1.3mm]
\leq & \,c_{1}\F^{(2)}_{\eps}(t)
\end{split}
\eee 
where $c_{1}:=2|\al|+\frac{4(\varrho+\mu+\lambda)}{\mu}+\frac{4(\varrho+\mu)}{\gamma}+4$. Combining (\ref{P1})-(\ref{P2}) with Young's inequality, we find 
\bee{}
\spl
&\quad\,\,\,\io (1+\eps A)^{-1}\Del u_{t}\cd \Del (1+\eps A)^{-1}\left\{t\psi_{1}+\psi_{2}\right\}\\[1mm]
&= t\io (1+\eps A)^{-1}\Del u_{t}\cd (1+\eps A)^{-1}\Del \psi_{1} +\io(1+\eps A)^{-1}\Del u_{t}\cd (1+\eps A)^{-1}\Del \psi_{2}\\[1mm]
&\leq 2\io |(1+\eps A)^{-1}\Del u_{t}|^2+ t^2\io |\Del \psi_{1}|^2+\io |\Del \psi_{2}|^2\\[1.7mm]
&\leq 4\F^{(2)}_{\eps}(t)+c_{2}
\end{split} 
\eee 
for all $t\in (0,\tm)$ where, according to (\ref{f})-(\ref{IN0}) and the assumption that  $\tm<\infty$,
$$ c_{2}:=\tm^2\io |\Del \psi_{1}|^2+\io |\Del \psi_{2}|^2<\infty.$$
Next, we estimate the integrals involving $f'(u)u_{t}$ and $f(u)$ in the right-hand side of (\ref{EE10.2}). To this end, we note that by (\ref{P1}) and Young's inequality 
\bee{}
\spl
&\,\,\,\io (1+\eps A)^{-1}\Del u_{t}\cd \Del (1+\eps A)^{-1}f'(u)u_{t}+\io (1+\eps A)^{-1}\Del u_{t}\cd \Del (1+\eps A)^{-1}f(u)
\\[0.5mm]
\leq&\, 4\F^{(2)}_{\eps}(t)+\io |\Del (f'(u)u_{t})|^2 +\io |\Del f(u)|^2\\[1mm]
\leq &\, 4\F^{(2)}_{\eps}(t)+\io \left|f''(u)\Del u u_{t}+f'''(u)|\na u|^2u_{t}+2f''(u)\na u\cd \na u_{t}+f'(u)\Del u_{t}\right|^2\\
&+ \io \left|f'(u)\Del u+f''(u)|\na u|^2\right|^2\\[1mm]
\leq &\, 4\F^{(2)}_{\eps}(t)+4\io [f''(u)]^2|\Del u|^2u^2_{t}+4\io [f'''(u)]^2|\na u|^4u^2_{t}+16 \io |f''(u)|^2|\na u|^2|\na u_{t}|^2\\
&+4\io |f'(u)|^2|\Del u_{t}|^2+ 2\io [f'(u)]^2|\Del u|^2+2\io [f''(u)]^2|\na u|^4\qquad\,\, \mbox{for all }\; t\in (0, \tm).
\end{split}
\eee 
Assumptions (\ref{E16.1}) and (\ref{f}) guarantee the existence of constants $c_{3},c_{4},c_{5}>0$ such that 
$$[f'(u)]^2\leq c_{3},\quad\,\, [f''(u)]^2\leq c_{4},\quad\,\, [f'''(u)]^2\leq c_{5} \quad \,\,\,\,\text{in}\,\,\, \Om \times (0, \tm).$$ Consequently 
\bee{}
\spl
&\,4\io [f''(u)]^2|\Del u|^2u^2_{t}+4\io [f'''(u)]^2|\na u|^4u^2_{t}+16 \io |f''(u)|^2|\na u|^2|\na u_{t}|^2+4\io |f'(u)|^2|\Del u_{t}|^2\\
&+ 2\io [f'(u)]^2|\Del u|^2+2\io [f''(u)]^2|\na u|^4\\[1mm]
\leq &\, 4c_{4}\io|\Del u|^2u^2_{t} +4c_{5}\io|\na u|^4u^2_{t}+16 c_{4}\io |\na u|^2|\na u_{t}|^2+4c_{3}\io |\Del u_{t}|^2+2c_{3}\io |\Del u|^2+2c_{4}\io |\na u|^4
\end{split}
\eee 
for all $t\in (0, \tm)$. At this point, due to elliptic regularity theory, the Gagliardo-Nirenberg inequality and the continuous embedding $W^{2,2} (\Om)\hookrightarrow L^{\infty}(\Om)$, there exist $c_{6}, c_{7}, c_{8}>0$ such that 
\bee
\|\vp \|^2_{L^\infty(\Om)}\leq c_{6}\|\Del \vp \|^2_{L^2(\Om)}\qquad \mbox{and} \qquad \|\na \vp\|^4_{L^4(\Om)}\leq c_{7}\|\Del \vp\|^2_{L^{2}(\Om)}\|\vp \|^2_{L^{\infty}(\Om)},
\eee 
together with 
\bee
\|\na \vp\|^2_{L^{4}(\Om)}\leq c_{8}\|\Del \vp\|_{L^{2}(\Om)}^2, 
\eee 
for all $\vp \in W^{2,2}(\Om)\cap W_{0}^{1,2}(\Om)$. Combining these estimates with (\ref{E16.1})-(\ref{E16.2}) and applying the Cauchy-Schwarz inequality we obtain
\bee{}
\spl
&4c_{4}\io|\Del u|^2u^2_{t} +4c_{5}\io|\na u|^4u^2_{t}+16 c_{4}\io |\na u|^2|\na u_{t}|^2+4c_{3}\io |\Del u_{t}|^2+2c_{3}\io |\Del u|^2+2c_{4}\io |\na u|^4\\
&\leq 4c_{4}c_{6}C_{1}\io |\Del u_{t}|^2+4c_{5}c_{6}c_{7}C_{1}K^2\io |\Del u_{t}|^2+16c_{4}\|\na u\|^2_{L^4(\Om)}\|\na u_{t}\|^2_{L^4(\Om)}+4c_{3}\io |\Del u_{t}|^2+2c_{3}C_{1}\\
&\quad +2c_{4}c_{7}C_{1}K^2\\[0.5mm]
&\leq c_{9} \io |\Del u_{t}|^2+c_{10}
\end{split}
\eee 
where $c_{9}:=4c_{4}c_{6}C_{1}+4c_{5}c_{6}c_{7}C_{1}K^2+16c_{4}\sqrt{c_{7}C_{1}}Kc_{8}+4c_{3}$ and $c_{10}:=2c_{3}C_{1}+2c_{4}c_{7}C_{1}K^2$. Putting together (\ref{EE10.2}) and the inequalities above, we arrive at
\bee
\frac{d}{dt}\F^{(2)}_{\eps}(t)+\frac{\bet}{2}\frac{d}{dt}\io |(1+\eps A)^{-1}\Del u|^2\leq (c_{1}+8)\F^{(2)}_{\eps}(t)+c_{9}\io |\Del u_{t}|^2+c_{10}+c_{2}
\eee 
for all $t\in (0, \tm)$. Integrating with respect to time and integrating by parts 
\bee
\spl
&\quad\,\,\,\F^{(2)}_{\eps}(t)+\frac{\bet}{2}\io |(1+\eps A)^{-1}\Del u|^2+\frac{\bet(c_{1}+8)}{2}\int_{0}^{t}e^{(c_1+8)(t-s)}\io |(1+\eps A)^{-1}\Del u|^2\,ds\\[0.5mm]
&\leq \F^{(2)}_{\eps}(0)+\frac{\bet}{2}\io |(1+\eps A)^{-1}\Del u_{0}|^2+c_{9}\int_{0}^{t}e^{(c_{1}+8)(t-s)}\io |\Del u_{t}|^2\,ds +(c_{10}+c_{2})\int_{0}^{t}e^{(c_1+8)(t-s)}\,ds. 
\end{split}
\eee 
Therefore, by (\ref{P1}), for all $\eps>0$ we have
\bee
\spl
\F^{(2)}_{\eps}(0)+\frac{\bet}{2}\io |(1+\eps A)^{-1}\Del u_{0}|^2=&\,\frac{\bet+\gamma}{2}\io |(1+\eps A)^{-1}\Del u_{0}|^2+\frac{1}{2}\io |(1+\eps A)^{-1}\Del u_{1}|^2\\[1mm]
\leq &\,\frac{\bet+\gamma}{2}\io |\Del u_{0}|^2+\frac{1}{2}\io |\Del u_{1}|^2.
\end{split}
\eee 
Estimate (\ref{E16.22}) now follows from (\ref{EE10.3}) and the same argument as in the proof of Lemma \ref{lem1.9}.
\qed  
The next bound in our regularity analysis can be carried out directly from (\ref{E15.5}), without requiring any further regularization.
\begin{lem}\label{lem11}
Assume that (\ref{IN0}) holds. If in addition $\tm<\infty$ and (\ref{E16.1}) holds for some $K>0$, then there exists $C_{3}>0$ such that 
\be{E111.1}
\io |\na u_{tt}(\cd,t)|^2\leq C_{3}\qquad \mbox{for a.e. } t\in (0, \tm).
\ee 
\end{lem}
\proof
In view of (\ref{Def1}) we infer that
\begin{align*}
\na u_{tt},\,\,\, \na u_{ttt}\quad &\quad \mbox{ belong to } \quad L_{loc}^{\infty}([0,\tm); L^2(\Om, \R^n)),\\
\Del u_{t},\,\,\, \Del u_{tt},\,\,\,  \Del u,\,\,\,\,  \Del v,\,\,\,\,\Del v_t &\quad\mbox{ belong to } \quad L_{loc}^{\infty}([0,\tm); L^2(\Om)).
\end{align*}
Consequently, choosing $\mu$ and $\lambda$ as in (\ref{E07}), we obtain 
\bee
\spl
\F^{(3)}(t):=&\frac{1}{2}\io |\na u_{tt}(\cd,t)|^2+\frac{\gamma}{2}\io |\Del u_{t}(\cd,t)|^2+\del \io \Del u(\cd, t) \Del u_{t}(\cd,t)+\varrho \io \Del v(\cd,t) \Del u_{t}(\cd,t)\\
&+\mu\io |\Del v(\cd, t)|^2+\lambda \io |\Del u(\cd,t)|^2 
\end{split}
\eee 
which belongs to $ W_{loc}^{1,\infty}([0,\tm))$ and, upon integrating by parts, satisfies 
\be{E111.2}
\spl
\frac{d}{dt}\F^{(3)}(t)=&\io \na u_{tt}\cd \na u_{ttt}+\gamma\io \Del u_{t} \Del u_{tt}+\varrho \io \Del u \Del u_{t}+\varrho \io \Del v \Del u_{tt}\\
&\,+\del \io |\Del u_{t}|^2+\del\io \Del u \Del u_{tt}+2\mu \io \Del u\Del v +2\lambda \io \Del u \Del u_{t}\\
=& \io \Del u_{tt}\cd\{-u_{ttt}+\gamma \Del u_{t}+\del \Del u+\varrho \Del v\}+\varrho \io \Del u \Del u_{t}+\del \io |\Del u_{t}|^2\\
&\,+2\mu \io \Del u \Del v+2\lambda \io \Del u\Del u_{t}
\end{split} 
\ee 
for a.e.\ $t\in (0,\tm)$. Moreover, in view of (\ref{E07}) and the Young's inequality,  
\be{E111.3}
\F^{(3)}(t)\geq \frac{1}{2}\io |\na u_{tt}|^2 +\frac{\gamma}{4}\io |\Del u_{t}|^2+\frac{\mu}{4}\io |\Del v|^2+\frac{\lambda}{4}\io |\Del u|^2
\ee  
for a.e.\ $t\in (0,\tm)$ which, together with a further use of Young's inequality, implies 
\be{E111.4}
\del \io |\Del u_{t}|^2+\varrho \io \Del u \Del u_{t}+2\mu \io \Del u\Del v+2\lambda \io \Del u\Del u_{t}\leq c_{1}\F^{(3)}(t)
\ee 
where $c_{1}:= \frac{4(\varrho +\lambda+\del)}{\gamma}+\frac{4(\varrho+\lambda+\mu)}{\lambda}+4$. Using (\ref{E15.5}) with integration by parts and Young's inequality
\be{E111.5}
\spl
&\,\,\io \Del u_{tt}\cd \{-u_{ttt}+\gamma \Del u_{t}+\del \Del u+\varrho \Del v\}\\[0.7mm]
=&\io \Del u_{tt}\cd \left\{\al u_{tt}+\bet u_{t}-f''(u)u^2_{t}-f'(u)u_{tt}-\psi_{1}\right\}\\[0.5mm]
=&- \al \io |\na u_{tt}|^2-\frac{\bet}{2}\frac{d}{dt}\io |\na u_{t}|^2+\io \na \{f''(u)u^2_{t}+f'(u)u_{tt}\}\cd \na u_{tt}
+\io \na \psi_{1}\cd \na u_{tt}\\[1mm]
\leq & [|\al|+1]\io |\na u_{tt}|^2-\frac{\bet}{2}\frac{d}{dt}\io |\na u_{t}|^2+\io f'''(u)u^2_{t}\na u\cd \na u_{tt}+2 \io u_{t}f''(u)\na u_{t}\cd \na u_{tt}\\
&+\io f''(u)u_{tt}\na u\cd \na u_{tt}+\io f'(u)|\na u_{tt}|^2 +\io |\na \psi_{1}|^2 \\[1mm]
& \leq [|\al|+4]\io |\na u_{tt}|^2-\frac{\bet}{2}\frac{d}{dt}\io |\na u_{t}|^2+\io |f'''(u)|^2u^{4}_{t}|\na u|^2+4 \io |f''(u)|^2u^2_{t}|\na u_{t}|^2\\
&+\io |f''(u)|^2u^2_{tt}|\na u|^2+\io f'(u)|\na u_{tt}|^2+\io |\na \psi_{1}|^2,\, \qquad \mbox{for a.e. } t\in (0, \tm).
\end{split}
\ee 
Under the assumption (\ref{E16.1}) and since $f\in C^4(\R)$, we may choose $c_{2}, c_{3}, c_{4}>0$ such that 
$$|f'''(u)|^2\leq c_{4},\quad \,\, |f''(u)|^2\leq c_{3},\quad\,\, |f'(u)|\leq c_{2},\quad \,\,\,\,\text{in } \Om \times (0, \tm).$$ Therefore
\be{E111.6}
\spl
&\io |f'''(u)|^2u^{4}_{t}|\na u|^2+4 \io |f''(u)|^2u^2_{t}|\na u_{t}|^2
+\io |f''(u)|^2u^2_{tt}|\na u|^2+\io f'(u)|\na u_{tt}|^2\\
\leq & c_{4}\io u^4_{t}|\na u|^2+4c_{3}\io u^2_{t}|\na u_{t}|^2+c_{3}\io u^2_{tt}|\na u|^2+c_{2}\io |\na u_{tt}|^2.
\end{split} 
\ee 
Elliptic regularity theory and the continuous embeddings $W^{1,2}(\Om)\hookrightarrow L^{4}(\Om)$,  $ W^{2,2}(\Om)\hookrightarrow L^{\infty}(\Om)$ and $W^{2,2}(\Om)\hookrightarrow W^{1,4}(\Om)$ ensure the existence of $c_{5}, c_{6}, c_{7}, c_{8}>0$ such that 
\be{II1}
\|\na \vp\|_{L^4(\Om)}\leq c_{5}\|\Del \vp\|_{L^{2}(\Om)}\qquad\, \mbox{for all } \; \vp \in W_{0}^{1,2}(\Om)\cap W^{2,2}(\Om),
\ee 
\be{II2}
\,\,\,\| \vp\|_{L^\infty(\Om)}\leq c_{6}\|\Del \vp\|_{L^{2}(\Om)}\qquad\, \mbox{for all } \; \vp \in W_{0}^{1,2}(\Om)\cap W^{2,2}(\Om),
\ee
and
\be{II3}
\| \vp\|_{L^4(\Om)}\leq c_{7}\|\na \vp\|_{L^{2}(\Om)}\qquad\, \mbox{for all } \; \vp \in W_{0}^{1,2}(\Om),
\ee 
as well as the Poincar\'e inequality 
\be{II4}
\| \na \vp\|_{L^2(\Om)}\leq c_{8}\|\Del \vp\|_{L^{2}(\Om)}\qquad\, \mbox{for all } \; \vp \in W_{0}^{1,2}(\Om)\cap W^{2,2}(\Om).
\ee 
Invoking Lemmas \ref{lem1.9}-\ref{lem1.10} together with  (\ref{II2}) and (\ref{II4}), there exists $c_{9}>0$ such that 
$$  c_{4}\io u^4_{t}|\na u|^2\leq c_{4}c^4_{6}c^2_{8}\|\Del u_{t}\|^4_{L^2(\Om)}\|\Del u\|^2_{L^2(\Om)}\leq c_{4}c^4_{6}c^2_{8} C_{1}C^2_{2}:=c_{9}.$$ 
Likewise, there exists $c_{10}>0$ such that
$$ 	4c_{3}\io u^2_{t}|\na u_{t}|^2\leq 4c_{3}c^2_{6}c^2_{8}\|\Del u_{t}\|^{4}_{L^2(\Om)}\leq  4c_{3}c^2_{6}c^2_{8} C^2_{2}:=c_{10}.$$ 
Moreover, the Cauchy-Schwarz inequality combined with (\ref{II3}) and Lemma \ref{lem1.9} yield
$$ c_{3}\io u^2_{tt}|\na u|^2\leq c_{3}\| u_{tt}\|^2_{L^{4}(\Om)} \|\na u\|^2_{L^{4}(\Om)}\leq c_{3}c^2_{7}c^2_{5}\|\na u_{tt}\|^2_{L^2(\Om)}\|\Del u \|^2_{L^2(\Om)}\leq c_{3}c^2_{7}c^2_{5} C_{1}\io |\na u_{tt}|^2.$$
Inserting these estimates into (\ref{E111.6}) and combining the result with (\ref{E111.4})-(\ref{E111.5}) and (\ref{E111.2}), there exist $c_{11},c_{12}>0$ such that 
\bee
\frac{d}{dt} \F^{(3)}(t)\leq c_{1}\F^{(3)}(t)+c_{11}\io |\na u_{tt}|^2+\io |\Del \psi_{1}|^2+c_{12}
\eee 
for a.e.\ $t\in (0,\tm)$. Since $$\psi_{1}= u_{3}+\al u_{2}+\bet u_{1}-\gamma \Del u_{1}-\del \Del u_{0}-f''(u_{0})u^2_{1}-f'(u_{0})u_{2}-f'(u_{0})u_{1}\in W^{2,2}(\Om)$$ by virtue of (\ref{IN0}), the Gronwall lemma ensures the boundedness of $\F^{(3)} $ on $(0,\tm)$. Finally,  applying  (\ref{E111.3}) once more, we arrive at the desired estimate (\ref{E111.1}).
\qed
To establish bounds for the remaining two quantities in (\ref{E14.1}), we shall employ equation (\ref{Def1E}) within an energy argument involving a functional analogous to that introduced in (\ref{E8}), combined with the  bounds obtained in Lemmas \ref{lem1.9}-\ref{lem11}. Once again, the limited regularity of the strong solutions necessitates the use of a suitable regularization procedure.
\begin{lem}\label{lem2.11}
Assume that (\ref{IN0}) holds. If in addition $\tm<\infty$ and (\ref{E16.1}) is satisfied for some $K>0$, then 
there exist $C_{4},C_{5}>0$ such that 
\be{E112.1}
\io |\Del u_{tt}(\cd,t)|^2\leq C_{4}
\ee 
and 
\be{E112.2}
\io |\na u_{ttt}(\cd,t)|^2\leq C_{5}
\ee 
for a.e.\ $t\in (0, \tm).$
\end{lem}
\proof
For $\eps >0$, let  $\mu$ and $\lambda$ be as in (\ref{E07}). We introduce the functional  
\be{E112.3}
\spl
\F^{(4)}_{\eps}(t):=\,&\frac{1}{2}\io |\na \left\{(1+\eps A)^{-1}u_{ttt}(\cd,t)\right\}|^2+\frac{\gamma}{2}\io |\Del \left\{(1+\eps A)^{-1}u_{tt}(\cd,t)\right\}|^2\\
&+\del \io \Del\left\{(1+\eps A)^{-1}u_{t}(\cd,t)\right\}\cd \Del\left\{(1+\eps A)^{-1}u_{tt}(\cd,t)\right\}
\\&+\varrho\io \Del\left\{(1+\eps A)^{-1}u(\cd,t)\right\}\cd \Del\left\{(1+\eps A)^{-1}u_{tt}(\cd,t)\right\}\\
&+\mu \io |\Del \left\{(1+\eps A)^{-1}u_{t}(\cd,t)\right\}|^2+\lambda \io |\Del \left\{(1+\eps A)^{-1}u(\cd,t)\right\}|^2
\end{split}
\ee
for a.e.\ $t\in (0, \tm)$. 
From the boundedness of the operators $\na (1+\eps A)^{-1}$ and $\Del (1+\eps A)^{-1}$ on $L^{2}(\Om)$ which follows from 
$$ \small \|\Del \{(1+\eps A)^{-1}\vp \}\|_{L^2(\Om)}\leq \frac{1}{\eps}\|\vp \|_{L^2(\Om)}\qquad \mbox{and}\qquad \|\na\{ (1+\eps A)^{-1}\vp\} \|_{L^2(\Om)}\leq \frac{1}{2\sqrt{\eps}}\|\vp \|_{L^2(\Om)} $$
for all $\vp \in L^2(\Om)$ and $\eps>0$, we infer from (\ref{Def1}) that 
$$\na\{(1+\eps A)^{-1}u_{ttt}\}, \na\{(1+\eps A)^{-1}u_{tttt}\}\in L_{loc}^{\infty}([0,\tm); L^2(\Om; \R^n))$$
and
$$\Del \{(1+\eps A)^{-1}u_{tt}\}, \, \Del \{(1+\eps A)^{-1}u_{ttt}\}, \,  \Del \{(1+\eps A)^{-1}u_{t}\}, \;\Del \{(1+\eps A)^{-1}u \}\in L_{loc}^{\infty}([0,\tm); L^2(\Om)).$$ 
Consequently $\F^{(4)}_{\eps}\in W_{loc}^{1,\infty}([0,\tm))$. 
Differentiating $\F^{(4)}_{\eps}$ with respect to time, integrating by parts and using (\ref{Def1E})-(\ref{P1}) we obtain
\allowdisplaybreaks 
\begin{align}\label{E112.5}
&\,\,\frac{d}{dt}\F^{(4)}_{\eps}(t)\nonumber\\[1mm]
={}&\io \na \{(1+\eps A)^{-1}u_{ttt}\}\cd
\na \{(1+\eps A)^{-1}u_{tttt}\}
+\gamma \io \Del \{(1+\eps A)^{-1}u_{tt}\}\cd
\Del \{(1+\eps A)^{-1} u_{ttt}\}
\nonumber\\
&+\del \io |\Del\{(1+\eps A)^{-1}u_{tt}\}|^2
+\del \io \Del\{(1+\eps A)^{-1}u_{t}\}\cd
\Del\{(1+\eps A)^{-1}u_{ttt}\}
\nonumber\\
&+\varrho \io \Del\{(1+\eps A)^{-1}u_{t}\}\cd
\Del\{(1+\eps A)^{-1}u_{tt}\}
+\varrho \io \Del\{(1+\eps A)^{-1}u\}\cd
\Del\{(1+\eps A)^{-1}u_{ttt}\}
\nonumber\\
&+2\mu \io \Del\{(1+\eps A)^{-1}u_{tt}\}\cd
\Del\{(1+\eps A)^{-1}u_{t}\}
+2\lambda \io \Del\{(1+\eps A)^{-1}u_{t}\}\cd
\Del\{(1+\eps A)^{-1}u\}
\nonumber\\
={}&\io \na \{(1+\eps A)^{-1}u_{ttt}\}\cd
\na \{(1+\eps A)^{-1}u_{tttt}\}
-\gamma \io \na \{(1+\eps A)^{-1}\Del u_{tt}\}\cd
\na \{(1+\eps A)^{-1}\Del u_{ttt}\}
\nonumber\\
&+\del \io |\Del\{(1+\eps A)^{-1}u_{tt}\}|^2
-\del \io \na\{(1+\eps A)^{-1}\Del u_{t}\}\cd
\Del\{(1+\eps A)^{-1}u_{ttt}\}
\nonumber\\
&+\varrho \io \Del\{(1+\eps A)^{-1}u_{t}\}\cd
\Del\{(1+\eps A)^{-1}u_{tt}\}
-\varrho \io \na \{(1+\eps A)^{-1}\Del u\}\cd
\na \{(1+\eps A)^{-1}u_{ttt}\}
\nonumber\\
&+2\mu \io \Del\{(1+\eps A)^{-1}u_{tt}\}\cd
\Del\{(1+\eps A)^{-1}u_{t}\}
+2\lambda \io \Del\{(1+\eps A)^{-1}u_{t}\}\cd
\Del\{(1+\eps A)^{-1}u\}
\nonumber\\
={}&\io \na \{(1+\eps A)^{-1}u_{ttt}\}\cd
\na \left\{(1+\eps A)^{-1}
\{u_{tttt}-\gamma \Del u_{tt}-\del \Del u_{t}
-\varrho \Del u\}\right\}
\nonumber\\
&+\varrho \io \Del\{(1+\eps A)^{-1}u_{t}\}\cd
\Del\{(1+\eps A)^{-1}u_{tt}\}
+2\mu \io \Del\{(1+\eps A)^{-1}u_{tt}\}\cd
\Del\{(1+\eps A)^{-1}u_{t}\}
\nonumber\\
&+2\lambda \io \Del\{(1+\eps A)^{-1}u_{t}\}\cd
\Del\{(1+\eps A)^{-1}u\}
\nonumber\\
={}&-\al\io |\na \{(1+\eps A)^{-1}u_{ttt}\}|^2
-\frac{\bet}{2}\frac{d}{dt}
\io |\na\{(1+\eps A)^{-1}u_{tt}\}|^2
\nonumber\\
&+\io \na \{(1+\eps A)^{-1}u_{ttt}\}\cd
\na \Big\{(1+\eps A)^{-1}
\left\{f'(u)u_{ttt}
+3f''(u)u_{t}u_{tt}
+f'''(u)u^3_{t}
+f'(u)u_{tt}
+f''(u)u^2_{t}\right\}\Big\}
\nonumber\\
&+\varrho \io \Del\{(1+\eps A)^{-1}u_{t}\}\cd
\Del\{(1+\eps A)^{-1}u_{tt}\}
+2\mu \io \Del\{(1+\eps A)^{-1}u_{tt}\}\cd
\Del\{(1+\eps A)^{-1}u_{t}\}
\nonumber\\
&+2\lambda \io \Del\{(1+\eps A)^{-1}u_{t}\}\cd
\Del\{(1+\eps A)^{-1}u\}
\end{align}
for a.e.\ $t\in(0,\tm)$. Furthermore, by virtue of (\ref{E07}) and Young's inequality, 
\be{E112.4}
\spl
\F^{(4)}_{\eps}(t)\geq&\, \frac{1}{2}\io |\na \{(1+\eps A)^{-1}u_{ttt}\}|^2+\frac{\gamma}{4}\io |\Del \{(1+\eps A)^{-1}u_{tt}\}|^2
+\frac{\mu}{4}\io |\Del \{(1+\eps A)^{-1}u_{t}\}|^2\\[1mm]&+\frac{\lambda}{4}\io |\Del \{(1+\eps A)^{-1}u\}|^2
\end{split}
\ee 
for a.e.\ $t\in (0,\tm)$.
Applying Young's inequality once more together with (\ref{E112.4}) we estimate
\begin{align}
\label{E112.6}
&-\al\io |\na \{(1+\eps A)^{-1}u_{ttt}\}|^2+\varrho \io \Del\{(1+\eps A)^{-1} u_{t}\}\cd \Del \{(1+\eps A)^{-1}u_{tt}\}\nonumber
\\&+2\mu \io \Del \{(1+\eps A)^{-1}u_{tt}\}\cd \Del \{(1+\eps A)^{-1}u_{t}\}
+2\lambda \io \Del \{(1+\eps A)^{-1}u_{t}\}\cd \Del \{(1+\eps A)^{-1}u\}\nonumber\\
\leq & |\al |\io |\na \{(1+\eps A)^{-1}u_{ttt}\}|^2+(\varrho+\mu)\io |\Del \{(1+\eps A)^{-1}u_{tt}\}|^2+(\mu+\lambda)\io |\Del \{(1+\eps A)^{-1}u_{t}\}|^2\nonumber\\
&+\lambda\io |\Del \{(1+\eps A)^{-1}u\}|^2\nonumber\\[1.4mm]
\leq& c_{1}\F^{(4)}_{\eps}(t)
\end{align}
where $c_{1}:= 2|\al|+\frac{4(\varrho+\mu)}{\gamma}+\frac{4(\mu+\lambda)}{\mu}+4.$ Another application of Young's inequality along with (\ref{P3}) yields 
\begin{align}
\label{E112.7}
&\io \na \{(1+\eps A)^{-1}u_{ttt}\}\cd \na \Big\{(1+\eps A)^{-1}\left\{f'(u)u_{ttt}+3f''(u)u_{t}u_{tt}+f'''(u)u^3_{t}+f'(u)u_{tt}+f''(u)u^2_{t}\right\}\Big\}\nonumber\\[1mm]
\leq& \io |\na \{(1+\eps A)^{-1}u_{ttt}\}|^2+\io |\na \{f'(u)u_{ttt}\}|^2+9\io |\na \{f''(u)u_{t}u_{tt}\}|^2+\io |\na \{f'''(u)u^3_{t}\}|^2\nonumber\\
&+\io |\na \{f'(u)u_{tt}\}|^2+\io |\na \{f''(u)u^2_{t}\}|^2\nonumber\\[1mm]
\leq& 2\F^{(4)}_{\eps}(t)+\io |\na \{f'(u)u_{ttt}\}|^2+9\io |\na \{f''(u)u_{t}u_{tt}\}|^2+\io |\na \{f'''(u)u^3_{t}\}|^2\nonumber\\
&+\io |\na \{f'(u)u_{tt}\}|^2+\io |\na \{f''(u)u^2_{t}\}|^2, \qquad \,\mbox{for a.e.}\; t\in (0,\tm).
\end{align}
By virtue of (\ref{E16.1}) and recalling that $f\in C^4(\R)$, there exist $c_{2},c_{3},c_{4},c_{5}>0$ such that 
$$ |f'(u)|^2\leq c_{2}, \quad\,\, |f''(u)|^2\leq c_{3}, \quad\,\, |f'''(u)|^2\leq c_{4}, \quad\,\, |f''''(u)|^2\leq c_{5},\qquad \,\mbox{in } \Om\times (0,\tm).$$
Therefore
\be{E112.8}
\spl
&\io |\na \{f'(u)u_{ttt}\}|^2+9\io |\na \{f''(u)u_{t}u_{tt}\}|^2+\io |\na \{f'''(u)u^3_{t}\}|^2\\
&+\io |\na \{f'(u)u_{tt}\}|^2+\io |\na \{f''(u)u^2_{t}\}|^2\\[0.7mm]
\leq&\, 2c_{3}\io |\na u|^2u^2_{ttt}+2c_{2}\io |\na u_{ttt}|^2+27c_{4} \io |\na u|^2|u_{t}|^2|u_{tt}|^2+27 c_{3}\io |\na u_{t}|^2|u_{tt}|^2\\&+27c_{3}\io |u_{t}|^2|\na u_{tt}|^2
+ 2c_{5}\io |\na u|^2|u_{t}|^{6}+18c_{4}\io |\na u_{t}|^2|u_{t}|^{4}+2c_{3}\io |\na u|^2|u_{tt}|^2
\\&+2c_{2}\io |\na u_{tt}|^2+2c_{4}\io |\na u|^2|u_{t}|^4
+8c_{3}\io |\na u_{t}|^2|u_{t}|^2
\end{split}
\ee 
for a.e.\ $t\in (0,\tm)$. From elliptic regularity theory and the embeddings $W^{2,2}(\Om)\hookrightarrow L^\infty(\Om)$, $W^{2,2}(\Om)\hookrightarrow W^{1,4}(\Om)$ and $W^{1,2}(\Om)\hookrightarrow L^4(\Om)$, we may choose $c_{6}, c_{7}, c_{8}, c_{9}>0$ such that 
\be{IM1}
\|\vp \|_{L^{\infty}(\Om)}\leq c_{6}\|\Del \vp\|_{L^2(\Om)} \qquad\,\,\,\hspace{0.6mm} \mbox{for all } \vp \in W^{2,2}(\Om)\cap W_{0}^{1,2}(\Om),
\ee 
\be{IM2}
\|\na \vp \|_{L^{4}(\Om)}\leq c_{7}\|\Del \vp\|_{L^2(\Om)} \qquad \mbox{for all } \vp \in W^{2,2}(\Om)\cap W_{0}^{1,2}(\Om)\vspace{0.5mm}
\ee 
\be{IM3}
\|\vp \|_{L^{4}(\Om)}\leq c_{8}\|\na \vp\|_{L^2(\Om)} \qquad \mbox{for all } \vp \in W_{0}^{1,2}(\Om)\qquad\quad\,\,\,\,\,
\ee 
whereas by the Poincar\'e inequality 
\be{IMM3}
\|\na \vp \|_{L^{2}(\Om)}\leq c_{9}\|\Del \vp\|_{L^2(\Om)} \qquad \mbox{for all } \vp \in W^{2,2}(\Om)\cap W_{0}^{1,2}(\Om).
\ee 
On the other hand,  Lemmas \ref{lem1.9}-\ref{lem11} ensure by virtue of (\ref{E16.1}) the existence of $C_{1},C_{2},C_{3}>0$ such that 
\be{IM4}
\|\Del u(\cd ,t)\|_{L^2(\Om)}\leq C_{1}, \qquad \hspace{0.5mm}\mbox{for a.e. } t\in (0,\tm),
\ee 
\be{IM5}
\|\Del u_{t}(\cd ,t)\|_{L^2(\Om)}\leq C_{2}, \qquad \mbox{for a.e. } t\in (0,\tm),
\ee 
as well as 
\be{IM6}
\|\na u_{tt}(\cd ,t)\|_{L^2(\Om)}\leq C_{3}, \qquad \mbox{for a.e. } t\in (0,\tm).
\ee 
Applying the Cauchy-Schwarz inequality together with (\ref{IM2})-(\ref{IM3}) and (\ref{IM4}) to the right hand side of (\ref{E112.8}) gives 
\bee
\spl
2c_{3}\io |\na u|^2u^2_{ttt}\leq 2c_{3}\|\na u\|^2_{L^4(\Om)}\| u_{ttt}\|^2_{L^4(\Om)}\leq 2c_{3}c^2_{7}c^2_{8}\|\Del u\|^2_{L^2(\Om)}\|\na u_{ttt}\|^2_{L^2(\Om)}
\leq c_{10}\io |\na u_{ttt}|^2
\end{split}
\eee 
where $c_{10}:= 2c_{3}c^2_{7}c^2_{8} C^2_{1}$. Similarly, employing
(\ref{IM1})-(\ref{IM3}) and (\ref{IM4})-(\ref{IM6})
\bee
\spl
27c_{4} \io |\na u|^2|u_{t}|^2|u_{tt}|^2\leq 27c_{4} & \|u_{t}\|^2_{L^{\infty}(\Om)}\|\na u\|^2_{L^{4}(\Om)}\|u_{tt}\|^2_{L^{4}(\Om)}\\[0.7mm]
\leq & 27c_{4} c^2_{6}c^2_{7}c^2_{8}\|\Del u\|^2_{L^2(\Om)}\|\Del u_{t}\|^2_{L^2(\Om)}\|\na u_{tt}\|^2_{L^2(\Om)}\\[0.7mm]
\leq & c_{11}
\end{split}
\eee 
where $c_{11}:= 27c_{4} c^2_{6}c^2_{7}c^2_{8} C^2_{1}C^2_{2}C^2_{3}$. 
Likewise, there exists $c_{12}>0$ such that 
\bee
\spl
&27 c_{3}\io |\na u_{t}|^2|u_{tt}|^2+27c_{3}\io |u_{t}|^2|\na u_{tt}|^2+ 2c_{5}\io |\na u|^2|u_{t}|^{6}
+18c_{4}\io |\na u_{t}|^2|u_{t}|^{4}\\&+2c_{3}\io |\na u|^2|u_{tt}|^2
+2c_{2}\io |\na u_{tt}|^2+2c_{4}\io |\na u|^2|u_{t}|^4
+8c_{3}\io |\na u_{t}|^2|u_{t}|^2\\[0.5mm]
\leq &\, 27c_{3}\|\na u_{t}\|^{2}_{L^4(\Om)}\|u_{tt}\|^{2}_{L^4(\Om)}+27c_{3}\|u_{t}\|^2_{L^\infty(\Om)}\|\na u_{tt}\|^2_{L^2(\Om)}\\
&+2c_{5}\|u_{t}\|^6_{L^{\infty}(\Om)}\|\na u\|^2_{L^2(\Om)}+18c_{4}\|u_{t}\|^4_{L^{\infty}(\Om)}\|\na u_{t}\|^2_{L^2(\Om)}+2c_{3}\|\na u\|^2_{L^{4}(\Om)}\|u_{tt}\|^2_{L^4(\Om)}\\
&+2c_{2}\|\na u_{tt}\|^2_{L^2(\Om)}+2c_{4}\|u_{t}\|^4_{L^{\infty}(\Om)}\|\na u\|^2_{L^2(\Om)}+8c_{3}\|\na u_{t}\|^2_{L^{4}(\Om)}\|u_{t}\|^2_{L^{4}(\Om)}\\
\leq &\, c_{12}.
\end{split}
\eee 
Inserting these estimates into (\ref{E112.8}) we obtain 
\be{E112.08}
\spl
&\quad\io |\na \{f'(u)u_{ttt}\}|^2+9\io |\na \{f''(u)u_{t}u_{tt}\}|^2\\[0.5mm]
&\quad+\io |\na \{f'''(u)u^3_{t}\}|^2
+\io |\na \{f'(u)u_{tt}\}|^2+\io |\na \{f''(u)u^2_{t}\}|^2\\
&\leq\, (2c_{2}+c_{10})\io |\na u_{ttt}|^2+c_{11}+c_{12} 
\end{split}
\ee
for a.e.\ $t\in (0,\tm)$. Consequently, combining (\ref{E112.5}) and (\ref{E112.6})-(\ref{E112.7}) with (\ref{E112.08}), we arrive at 
\bee
\frac{d}{dt}\F^{(4)}_{\eps}(t)+\frac{\bet}{2}\frac{d}{dt}\io |\na \{(1+\eps A)^{-1}u_{tt}\}|^2\leq c_{13}\F^{(4)}_{\eps}(t)+c_{14}\io |\na u_{ttt}|^2+c_{15}
\eee 
for a.e.\ $t\in (0,\tm)$, where $c_{13}:=c_{1}+2$, $c_{14}:=  2c_{2}+c_{10}$ and $c_{15}:= c_{11}+c_{12}$. Integrating with respect to time and then integrating by parts, we obtain 
\be{E112.20}
\spl 
&\,\F^{(4)}_{\eps}(t)+\frac{\bet}{2}\io |\na \{(1+\eps A)^{-1}u_{tt}\}|^2+\frac{\bet c_{13}}{2}\int_{0}^{t}e^{c_{13}(t-s)}\io |\na \{(1+\eps A)^{-1}u_{tt}\}|^2\\[0.5mm]\leq& \F^{(4)}_{\eps}(0)+\frac{\bet}{2}\io |\na \{(1+\eps A)^{-1}u_{2}\}+c_{14}\int_{0}^{t}e^{c_{13}(t-s)}\io |\na u_{ttt}|^2+c_{15}\int_{0}^{t}e^{c_{13}(t-s)}\,ds\\[0.5mm]
\leq& c_{16}+c_{17}\int_{0}^{t}\io |\na u_{ttt}|^2
\end{split} 
\ee 
for a.e.\ $t\in (0, \tm)$, where 
$$c_{16}:=\sup_{\eps>0}\F^{(4)}_{\eps}(0)+\sup_{\eps>0}\frac{\bet}{2}\io |\na \{(1+\eps A)^{-1}u_{2}\}|^2+\frac{c_{15}}{c_{13}}e^{c_{13}\tm}$$ and $c_{17}:= c_{14}e^{c_{13}\tm}$. Both these constants are finite due to the assumption $\tm<\infty$ and the fact that, by (\ref{P1})-(\ref{P3}) and the Cauchy-Schwarz inequality,
\bee
\spl
&\quad\,\F^{(4)}_{\eps}(0)+\frac{\bet}{2}\io |\na \{(1+\eps A)^{-1}u_{2}\}|^2\\[0.5mm]
&\leq \frac{1}{2}\io |\na u_{3}|^2+\left(\frac{\gamma}{2}+\del+\varrho\right)\io |\Del u_{2}|^2+(\del+\mu) \io |\Del u_{1}|^2+(\varrho+\lambda) \io |\Del u_{0}|^2+\frac{\bet}{2} \io |\na u_{2}|^2
\end{split}
\eee
for all $\eps >0$. Finally, by (\ref{P3})-(\ref{P4}) along with (\ref{P1})-(\ref{P2}) and (\ref{Def1}), we have 
$$ \na \{(1+\eps A)^{-1}u_{ttt}\}\rightharpoonup \na u_{ttt} \quad \mbox{in}\; L^2(\Om; \R^n)\qquad \mbox{ and }\qquad \Del \{(1+\eps A)^{-1}u_{tt}\}\rightharpoonup \Del u_{tt}\quad \mbox{in}\; L^2(\Om)$$
as $\eps \searrow 0$ for a.e. $t\in (0, \tm)$. Next, the lower semicontinuity of $L^2(\Om)$-norm under weak convergence together with (\ref{E112.20}) and (\ref{E112.4}) imply that the function 
$$ y(t):=\frac{1}{2}\io |\na u_{ttt}(\cd,t)|^2+\frac{\gamma}{4}\io |\Del u_{tt}(\cd,t)|^2 \qquad t\in (0, \tm)$$
fulfills 
\bee
\spl
y(t)\leq \liminf_{\eps\searrow 0} \F^{(4)}_{\eps}(t)
\leq  c_{16}+c_{17}\int_{0}^{t}\io |\na u_{ttt}|^2
\leq c_{16}+2c_{17}\int_{0}^{t}y(s)\,ds, \qquad\, \mbox{for a.e. } t\in (0, \tm).
\end{split}
\eee 
Gronwall's inequality and $\tm<\infty$ yield $y\in L^{\infty}([0,\tm))$, and (\ref{E112.1})-(\ref{E112.2}) follow.
\qed 
Invoking  Lemma \ref{lem2.4},  Lemma \ref{lem2.5} and Lemmas \ref{lem1.9}-\ref{lem2.11},
the proof of Theorem \ref{Th1} is finished.
%
\section{Finite-time blow-up:\ Proof of Theorem \ref{Th1.3}}
The analysis of the finite-time blow-up for (\ref{0}) begins with the following observation on the evolution of the projection of the solution onto the first eigenfunction. Once more, our analysis is based on the integrated equation (\ref{E15.07}).
\begin{lem}\label{lem5.1}
Suppose that (\ref{f})-(\ref{IN0}) hold, and let $u$ be the strong solution given by Theorem \ref{Th1}. Set 
\be{E114.1}
\kappa :=\io e_{1}
\ee 
and define 
\be{E114.2}
y(t):=\frac{1}{\kappa}\io u(\cd,t)e_{1}(\cd), \qquad t\in [0,\tm).
\ee 
Then  $y\in C^1([0,\tm))$ and 
\be{E114.3}
\spl
y'(t)+\al y(t)=&-(\bet+\lambda_{1}\gamma)\int_{0}^{t}y(s)\,ds-\del\lambda_{1}\int_{0}^{t}\int_{0}^{s}y(\sigma)\,d\sigma ds-\varrho \lambda_{1}\int_{0}^{t}\int_{0}^{s}\int_{0}^{\sigma}y(\xi)\,d\xi d\sigma ds\\
&+\frac{1}{\kappa}\io f(u)e_{1}+\frac{1}{\kappa}\io F(u)e_{1}+\frac{t^2}{2\kappa}\io \psi_{1}e_{1}+\frac{t}{\kappa}\io \psi_{2}e_{1}+\io \psi_{3} e_{1}
\end{split}
\ee 
where $\psi_{1},\psi_{2},\psi_{3}$ are defined in  (\ref{E15.2}), (\ref{E15.3}), (\ref{E15.4}) respectively, while  $F$ is defined in (\ref{E15.01}).
\end{lem}
\proof
The asserted regularity  of $y$ follows directly from (\ref{T1}). Next, using (\ref{E15.07}), we obtain 
\be{E114.4}
\spl
y'(t)+\al y(t)=&\,\frac{1}{\kappa}\io \left\{-\bet v+\gamma \Del v+\del \Del w+\varrho \Del z+f(u)+F(u)+\frac{t^2}{2}\psi_{1}+t\psi_{2}+\psi_{3}\right\}\cd e_{1}\\
=& -\frac{\bet}{\kappa}\io\left\{ \int_{0}^{t}u(\cd,s)\,ds\right\}\cd e_{1}+\frac{\gamma}{k}\io \Del v\cd e_{1}+\frac{\del}{\kappa}\io \Del w\cd e_{1}+\frac{\varrho}{\kappa}\io \Del z\cd e_{1}\\
&+\frac{1}{\kappa}\io f(u)e_{1}+\frac{1}{\kappa}\io F(u)e_{1}+\frac{t^2}{2\kappa}\io \psi_{1} e_{1}+\frac{t}{\kappa}\io \psi_{2}e_{1}+\io \psi_{3} e_{1}
\end{split}
\ee 
{for all} $t\in (0,\tm)$.
Since $-\Del e_{1}=\lambda_{1}e_{1}$ in $\Om$ and $e_{1}=0$ on $\pa\Om$, integrating by parts and using the definitions of  $v$, $w$ and $z$ we have
$$ \frac{\gamma}{\kappa}\io \Del v\cd e_{1}=-\frac{\lambda_{1}\gamma}{k}\io  v e_{1}=-\frac{\lambda_{1}\gamma}{\kappa}\io  \left\{\int_{0}^{t}u(\cd,s)\,ds\right\}\cd  e_{1}=-\lambda_{1}\gamma \int_{0}^{t}y(s)\,ds,$$
$$ \frac{\del }{\kappa}\io \Del w\cd e_{1}=-\frac{\lambda_{1}\del}{\kappa}\io  w e_{1}=-\frac{\lambda_{1}\del}{\kappa}\io  \left\{\int_{0}^{t}\int_{0}^{s}u(\cd,\sigma)\,d\sigma ds\right\}\cd e_{1}=-\lambda_{1}\del \int_{0}^{t}\int_{0}^{s}y(\sigma)\,d\sigma ds,$$
and
$$ \frac{\varrho}{\kappa}\io \Del z\cd e_{1}=-\frac{\lambda_{1}\varrho}{\kappa}\io  \left\{\int_{0}^{t}\int_{0}^{s}\int_{0}^{\sigma}u(\cd,\xi)\,d\xi d\sigma ds\right\} \cd e_{1}=-\lambda_{1}\varrho \int_{0}^{t}\int_{0}^{s}\int_{0}^{\sigma}y(\xi)\,d\xi d\sigma ds$$
for all $t\in (0,\tm)$. Inserting these identities into (\ref{E114.4}) we arrive at (\ref{E114.3}).
\qed
Within the additional assumptions on $f$ imposed in Theorem \ref{Th1.3}, and for suitably chosen initial data, the integro-differential equation (\ref{E114.3}) can be reduced to a first-order ordinary differential inequality whose validity necessarily breaks down within a sufficiently short time. 
\begin{lem}\label{lem1.15}
Assume that (\ref{f}) and (\ref{T1.3.1})-(\ref{T1.3.3}) hold. For every $T_{0}>0$, there is $K_{0}=K_{0}(T_{0})>0$ such that whenever $u_{0}\in W^{4,2}(\Om)\cap W_{0}^{1,2}(\Om)$ satisfies (\ref{T1.3.4}), one can find $K_{1}=K_{1}(u_{0})>0$  with the property that if $u_{1}\in W^{4,2}(\Om)\cap W_{0}^{1,2}(\Om)$ fulfils (\ref{T1.3.5}), we may choose $K_{2}=K_{2}(u_{1})>0$ such that for every $u_{2}\in W^{2,2}(\Om)\cap W_{0}^{1,2}(\Om)$ obeying (\ref{T1.3.6}), there exists $K_{3}=K_{3}(u_{2})>0$  such that for each $u_{3}\in W^{2,2}(\Om)\cap W_{0}^{1,2}(\Om)$ satisfying (\ref{T1.3.7}) the corresponding strong solution $u$ to (\ref{0}) given by Theorem \ref{Th1} verifies
\be{E116.01}
\tm\leq T_{0}.
\ee 
\end{lem}
\proof
Fix $T_{0}>0$. In view of  (\ref{T1.3.3}), we may select $\xi_{1}=\xi_{1}(T_{0})>\xi_{0}$ such that 
\be{E116.1}
\int_{\xi_{1}}^{\infty}\frac{d\xi}{f(\xi)}\leq \frac{T_{0}}{4}.
\ee 
Assumption (\ref{T1.3.2}) furthermore allows us to choose $\xi_{2}=\xi_{2}(T_{0})> \xi_{0}$ such that
\be{E116.3}
\qquad\qquad\qquad\,\,\,\,\,\, |\al |\cd\frac{\xi}{f(\xi)}\leq \frac{1}{8}\qquad \mbox{for all }\, \xi\geq \xi_{2},
\ee 
\be{E116.4}
|\bet+\lambda_{1}\gamma-f'(0)|T_{0}\cd\frac{\xi}{f(\xi)}\leq \frac{1}{8}\qquad \mbox{for all }\, \xi\geq \xi_{2},
\ee 
\be{E116.5}
\frac{\del\lambda_{1}T^2_{0}}{2}\cd\frac{\xi}{f(\xi)}\leq \frac{1}{8}\qquad \mbox{for all }\, \xi\geq \xi_{2},
\ee 
along with
\be{E116.55}
\frac{\varrho\lambda_{1}T^3_{0}}{6}\cd\frac{\xi}{f(\xi)}\leq \frac{1}{8}\qquad \mbox{for all }\, \xi\geq \xi_{2}.
\ee 
Then we define 
\be{E116.6}
K_{0}\equiv K_{0}(T_{0}):=\kappa\max\{\xi_{1}, \xi_{2}\}.
\ee 
Assume now that $u_{0}\in W^{4,2}(\Om)\cap W^{1,2}_{0}(\Om)$ satisfies (\ref{T1.3.4}) and choose $K_{1}=K_{1}(u_{0})>0$ such that 
\be{E116.7}
K_{1}> -\al \io u_{0}e_{1}+\io f(u_{0})e_{1}.
\ee 
Let $u_{1}\in W^{4,2}(\Om)\cap W_{0}^{1,2}(\Om)$ fulfilling (\ref{T1.3.5}). We fix $K_{2}=K_{2}(u_{1})>0$ so that 
\be{E116.08}
K_{2}> -\al \io u_{1}e_{1}-\bet \io u_{0}e_{1}+\gamma\io \Del u_{0}e_{1}+\io f'(u_{0})u_{1}e_{1}+\io f(u_{0})e_{1}.
\ee 
Next, let $u_{2}\in W^{2,2}(\Om)\cap W_{0}^{1,2}(\Om)$ obeying (\ref{T1.3.6}). We select $K_{3}=K_{3}(u_{2})>0$ satisfying
\be{E116.8}
\spl
K_{3}>& -\al\io u_{2}e_{1}-\bet \io u_{1}e_{1}+\gamma\io  \Del u_{1}e_{1}+\del\io  \Del u_{0}e_{1}+\io f''(u_{0})u^2_{1}e_{1}\\&+\io f'(u_{0})u_{2}e_{1}+\io f'(u_{0})u_{1}e_{1}.
\end{split}
\ee 
Finally, suppose that $u_{3}\in W^{2,2}(\Om)\cap W_{0}^{1,2}(\Om)$ verifies (\ref{T1.3.7}).
We argue that the corresponding strong solution $u$ enjoys the property stated in (\ref{E116.01}). To this end, assume by contradiction that
$\tm>T_{0}$, and let $y$ be as in Lemma \ref{lem5.1}. From (\ref{T1.3.5}) and (\ref{E116.7}), the term involving $\psi_{3}$  on the right-hand side of (\ref{E114.3}) can be estimated as 
\bee 
\spl
\io \psi_{3} e_{1}=\io u_{1}e_{1}+\al \io u_{0}e_{1}-\io f(u_{0})e_{1}
\geq \,K_{1}+\al \io u_{0}e_{1}-\io f(u_{0})e_{1}>0.
\end{split}
\eee 
Similarly, by (\ref{T1.3.6}) and (\ref{E116.08}),
\bee 
\spl
\io \psi_{2} e_{1}=&\io u_{2}e_{1}+\al \io u_{1}e_{1}+\bet \io u_{0}e_{1}-\gamma \io \Del u_{0}e_{1}-\io f'(u_{0})u_{1}e_{1}-\io f(u_{0})e_{1}\\
\geq& \,K_{2}+\al \io u_{1}e_{1}+\bet \io u_{0}e_{1}-\gamma \io \Del u_{0}e_{1}-\io f'(u_{0})u_{1}e_{1}-\io f(u_{0})e_{1}>0.
\end{split}
\eee 
Likewise, (\ref{T1.3.7}) together with (\ref{E116.8}) yield
\bee 
\spl
\io \psi_{1}e_{1}=&\io u_{3}e_{1}+\al \io u_{2}e_{1}+\bet \io u_{1}e_{1}-\gamma \io \Del u_{1}e_{1}-\del \io \Del u_{0}e_{1}-\io f''(u_{0})u^2_{1}e_{1}\\
&-\io f'(u_{0})u_{2}e_{1}-\io f'(u_{0})u_{1}e_{1}\\
\geq &\, K_{3}+ \,\al \io u_{2}e_{1}+\bet \io u_{1}e_{1}-\gamma \io \Del u_{1}e_{1}-\del \io \Del u_{0}e_{1}-\io f''(u_{0})u^2_{1}e_{1}\\
&-\io f'(u_{0})u_{2}e_{1}-\io f'(u_{0})u_{1}e_{1}>0.
\end{split}
\eee 
Consequently, from (\ref{E114.3}), we deduce that
\be{E114.20}
\spl
y'(t)>& -\al y(t)-(\bet+\lambda_{1}\gamma)\int_{0}^{t}y(s)\,ds-\del\lambda_{1}\int_{0}^{t}\int_{0}^{s}y(\sigma)\,d\sigma ds-\varrho \lambda_{1}\int_{0}^{t}\int_{0}^{s}\int_{0}^{\sigma}y(\xi)\,d\xi d\sigma ds\\
&+\frac{1}{\kappa}\io f(u)e_{1}+\frac{1}{\kappa}\io F(u)e_{1}
\end{split}
\ee 
for all $t\in (0,\tm)$. At this point, the convexity assumption (\ref{T1.3.1}) and the nonnegativity of $e_{1}$, together with the normalization $\io \frac{e_{1}}{\kappa}=1$ from (\ref{E114.1}), allow us to apply the Jensen's inequality 
\be{E1R}
f(y(t))=f\left(\io u(\cd,t)\cd \frac{e_{1}}{\kappa}\right)\leq \frac{1}{\kappa}\io f(u(\cd,t))e_{1},\qquad\, \mbox{for all }\, t\in (0, \tm).
\ee
Moreover, the convexity of $f$ and the fact that $f(0)=0$ give
\be{E116.30}
f(\xi)\geq f'(0)\xi\qquad \mbox{for all }\, \xi \in \R.
\ee 
Therefore, from (\ref{E1R}) we infer that
$$ \frac{1}{\kappa}\io F(u)e_{1}=\int_{0}^{t}\frac{1}{k}\io f(u(\cd,s))e_{1}\,ds\geq \int_{0}^{t}f(y(s))\,ds\geq f'(0)\int_{0}^{t}y(s)\,ds.$$
Substituting these estimates into (\ref{E114.20}) we have
\be{E114.21}
\spl
y'(t)>& -\al y(t)-(\bet+\lambda_{1}\gamma-f'(0))\int_{0}^{t}y(s)\,ds-\del\lambda_{1}\int_{0}^{t}\int_{0}^{s}y(\sigma)\,d\sigma ds\\
&-\varrho \lambda_{1}\int_{0}^{t}\int_{0}^{s}\int_{0}^{\sigma}y(\xi)\,d\xi d\sigma ds
+f(y(t))
\end{split}
\ee 
for all $t\in (0,\tm)$.  We now claim that
\be{E114.22}
y'(t)\geq \frac{1}{2}f(y(t))\qquad \mbox{for all }\, t\in (0,T_{0}).
\ee 
To establish this, let
\be{E114.23}
\wh T_{0}:=\sup\Big\{T\in (0, T_{0}): \, y(t)>\frac{K_{0}}{\kappa} \;\; \mbox{and}\;\; y'(t)>0\,\, \mbox{for all}\,\, t\in (0,T)\Big\}.
\ee 
This set is non-empty and  $\wh T_{0}>0$. Indeed  $y,y'$ are continuous on $[0,\tm)$ by Lemma \ref{lem5.1} whereas from (\ref{T1.3.4}) we have  $y(0) >\frac{K_{0}}{\kappa}$ and as a direct consequence of (\ref{E116.3}) we also have 
\bee
y'(0)>f(y(0))-\al y(0)=f(y(0))\left\{1-\frac{\al y(0)}{f(y(0))}\right\}\geq  f(y(0))\cd \frac{7}{8}>0.
\eee 
Now, on  $(0,\wh T_{0})$, we have  $y\geq \frac{K_{0}}{\kappa}\geq \xi_{2}$ by  (\ref{E114.23}) and (\ref{E116.6}). Hence, from (\ref{E116.3}), we find
\be{E116.400}
\frac{\al y(t)}{f(y(t))}\leq |\al |\cd \frac{y(t)}{f(y(t))}\leq \frac{1}{8}\qquad\, \mbox{for all }\, t\in (0,\wh T_{0}).
\ee 
Likewise, using  $y'\geq 0$ on $(0, \wh T_{0})$ together with (\ref{E116.4}), for all $t\in (0,\wh T_{0})$ we obtain 
\be{E116.40}
\spl
\big\{\bet+\lambda_{1}\gamma-f'(0)\big\}\cd\frac{\int_{0}^{t}y(s)\,ds}{f(y(t))}\leq& |\bet+\lambda_{1}\gamma-f'(0)\big|\cd \frac{ty(t)}{f(y(t))}\\[0.7mm]
\leq & |\bet+\lambda_{1}\gamma-f'(0)\big|T_{0}\cd \frac{y(t)}{f(y(t))}\\[0.7mm]
\leq & \frac{1}{8}
\end{split}
\ee 
since $\wh T_{0}\leq T_{0}$. By the same token, (\ref{E116.5}) yields
\be{E116.41}
\spl
\del \lambda_{1}\cd\frac{\int_{0}^{t}\int_{0}^{s}y(\sigma)\,d\sigma ds}{f(y(t))}\leq \del \lambda_{1}\cd  \frac{\frac{t^2}{2}y(t)}{f(y(t))}
\leq  \frac{\del \lambda_{1}T^2_{0}}{2}\cd \frac{y(t)}{f(y(t))}
\leq  \frac{1}{8}
\end{split}
\ee 
for all $t\in (0,\wh T_{0})$. Finally, from (\ref{E116.55}), 
\bee
\spl
\varrho\lambda_{1}\cd\frac{\int_{0}^{t}\int_{0}^{s}\int_{0}^{\sigma}y(\xi)\,d\xi d\sigma ds}{f(y(t))}\leq& \varrho \lambda_{1}\cd  \frac{\frac{t^3}{6}y(t)}{f(y(t))}
\leq  \frac{\varrho \lambda_{1}T^3_{0}}{6}\cd \frac{y(t)}{f(y(t))}
\leq  \frac{1}{8}
\end{split}
\eee 
for all $t\in (0,\wh T_{0})$. Combining the latter inequality with (\ref{E114.21}) and (\ref{E116.400})-(\ref{E116.41}), we arrive at
\be{E116.42}
y'(t)\geq \frac{1}{2}f(y(t))\qquad \mbox{for all }\, t\in (0,\wh T_{0}).
\ee 
Since $f$ is positive on $\big[\frac{K_{0}}{\kappa}, \infty)$, this inequality in particular shows that $\wh T_{0}=T_{0}$. Indeed, if $\wh T_{0}<T_{0}$, the continuity of $y'$ would imply that either  $y(\wh T_{0})=\frac{K_{0}}{\kappa}$ or $y'(\wh T_{0})=0$. But from (\ref{E116.42}) and the positivity of $f$ we would have both 
$y(\wh T_{0})>\frac{K_{0}}{\kappa}$ and $y'(\wh T_{0})>0$, which is a contradiction. Thus (\ref{E114.22}) follows from (\ref{E116.42}).  Integrating (\ref{E114.22}) on $(0,T_{0})$ and using $y(0)\geq \frac{K_{0}}{\kappa}$, we finally obtain 
\bee
\frac{T_{0}}{2}\leq \int_{y(0)}^{y(T_{0})}\frac{d\xi}{f(\xi)}\leq \int_{\frac{K_{0}}{\kappa}}^{\infty}\frac{d\xi}{f(\xi)}.
\eee 
By the definition  of $K_{0}$ and in view of (\ref{E116.1}), we now infer that 
\bee
\frac{T_{0}}{2}\leq \int_{\xi_{1}}^{\infty}\frac{d\xi}{f(\xi)}\leq \frac{T_{0}}{4}
\eee 
which is impossible. Therefore $\tm \leq T_{0}$, as desired. 
\qed
Exploiting Lemma \ref{lem1.15} and Theorem \ref{Th1}, we finish the proof of Theorem \ref{Th1.3}.

\section*{Acknowledgments} The first author thanks Professor Michael Winkler for fruitful discussions on the topic of this paper.
The second author is a member of the Gruppo Nazionale per l’Analisi Matematica,
la Probabilità e le loro Applicazioni (GNAMPA) of the Istituto Nazionale di Alta Matematica (INdAM).

\section*{Declarations}
{\bf Conflict of interest statement.} 
The authors declare that they have no conflict of interest.

\medskip
{\bf Data availability statement.} 
Data sharing is not applicable to this article as no datasets were
generated or analyzed during the study.



\begin{thebibliography}{99}
%

\bibitem{Benabbas2026}
Benabbas, I., Said-Houari, B., {\it Local well-posedness of a coupled Jordan--Moore--Gibson--Thompson--Pennes model of nonlinear ultrasonic heating}, Nonlinear Anal. Real World Appl. {\bf 90} (2026), 104546.

\bibitem{Black} Black, T., Winkler, M., {\it Global solutions and large time stabilization in a model for thermoacoustics in a standard linear solid}, Evol. Equ. Control Theory {\bf 25} (2026), 278--308.

\bibitem{Bon2} 
Bongarti, M., Charoenphon, S., Lasiecka, I., {\it Singular thermal relaxation limit for the Moore--Gibson--Thompson equation}. In: Semigroups of Operators:\ Theory and Applications (SOTA 2018), pp. 147--182. Springer (2019).

\bibitem{Bon1}
Bongarti, M., Charoenphon, S., Lasiecka, I., {\it Vanishing relaxation time dynamics of the Jordan--Moore--Gibson--Thompson equation arising in nonlinear acoustics}, \mbox{J. Evol. Equ.} {\bf 21} (2021),  3553--3584.

\bibitem{Caixeta2016}	Caixeta, A.H., Lasiecka, I., Domingos Cavalcanti, N.V., {\it On long time behavior of Moore--Gibson--Thompson equation with molecular relaxation}, Evol. Equ. Control Theory {\bf 5} (2016), 661--676.

\bibitem{Chand} Chandrasekharaiah, D.S., {\it Hyperbolic thermoelasticity:\ A review of recent literature}, Appl. Mech. Rev. {\bf 51} (1998), 705--729.

\bibitem{Chen2021}	Chen, W., Ikehata, R., {\it The Cauchy problem for the Moore--Gibson--Thompson equation in the dissipative case}, J. Differential Equations {\bf 292} (2021), 176--219.

\bibitem{Chen2026}	Chen, W., Liu, Y., Palmieri, A., Qin, X., {\it The influence of viscous dissipations on the nonlinear acoustic wave equation with second sound}, Nonlinear Anal. {\bf 273} (2026), 114220.

\bibitem{Chen2} Chen, W., Palmieri, A., {\it Nonexistence of global solutions for the semilinear Moore--Gibson--Thompson equation in the conservative case}, Discrete Contin. Dyn. Syst. {\bf 40} (2020), 5513--5540.

\bibitem{Chen1} Chen, W., Palmieri, A., {\it A blow-up result for the semilinear Moore--Gibson--Thompson equation with nonlinearity of derivative type in the conservative case}, Evol. Equ. Control Theory {\bf 10} (2021), 673--687.

\bibitem{Cla} Claes, L., Winkler, M., {\it  Local strong solutions in a quasilinear Moore-Gibson-Thompson type model for thermoviscoelastic evolution in a standard linear solid}, Nonlinear Differ. Equ. Appl. NoDEA {\bf 33}  (2026), 102.

\bibitem{Cleveland2012}
Cleveland, R.O., McAteer, J.A., Physics of shock-wave lithotripsy. In: {Smith's Textbook of Endourology}, pp. 527--558. Wiley Online Library (2012).

\bibitem{DeK} Dekkers, A., Rozanova-Pierrat, A., {\it Cauchy problem for the Kuznetsov equation}, Discrete Contin. Dyn. Syst. {\bf 39} (2019), 277--307.

\bibitem{Dell3} Dell'Oro, F., Lasiecka, I., Pata, V., {\it The Moore--Gibson--Thompson equation with memory in the critical case}, 
J. Diﬀerential Equations {\bf 261} (2016), 4188--4222.

\bibitem{Dell2} Dell'Oro, F., Pata, V., {\it On a fourth order equation of Moore--Gibson--Thompson type}, Milan J. Math. {\bf 85} (2017), 215--234.

\bibitem{Dell1} Dell'Oro, F., Pata, V., {\it A hierarchy of heat conduction laws}, Discrete Contin. Dyn. Syst. Ser. S {\bf 16} (2023), 2636--2648.

\bibitem{Fila1994}
Fila, M., Lieberman, G.M., {\it Derivative blow-up and beyond for quasilinear parabolic equations}, {Differential Integral Equations} \textbf{7} (1994), 811--821.

\bibitem{Frik} Fricke, T., Winkler, M., {\it Classical solutions to a model for heat generation during acoustic wave propagation in a standard linear solid}, Commun. Math. Anal. Appl. {\bf 5} (2026), 91--114.

\bibitem{Green1} Green, A.E., Naghdi, P.M., {\it On undamped heat waves in an elastic solid},  J. Therm. Stresses {\bf 15} (1992), 253--264.

\bibitem{Green2} Green, A.E., Naghdi, P.M., {\it Thermoelasticity without energy dissipation}, J. Elast. {\bf 31} (1993), 189--208.

\bibitem{Kaltenbacher2012}
Kaltenbacher, B., Lasiecka, I., Pospieszalska, M.K., {\it Well-posedness and exponential decay of the energy in the nonlinear Jordan--Moore--Gibson--Thompson equation arising in high intensity ultrasound}, Math. Models Methods Appl. Sci. {\bf 22} (2012), 1250035.

\bibitem{Kaltenbacher2019}
Kaltenbacher, B., Nikoli\'{c}, V., {\it The Jordan--Moore--Gibson--Thompson equation: well-posedness with quadratic gradient nonlinearity and singular limit for vanishing relaxation time}, Math. Models Methods Appl. Sci. {\bf 29} (2019), 2523--2556.

\bibitem{KaltenbacherNikolic2021}
Kaltenbacher, B., Nikoli\'{c}, V., {\it The inviscid limit of third-order linear and nonlinear acoustic equations}. SIAM J. Appl. Math. {\bf 81}  (2021), 1461--1482.

\bibitem{Las3} Lasiecka, I., {\it Global solvability of Moore--Gibson--Thompson equation with memory arising in nonlinear acoustics}, J. Evol. Equ. {\bf 17} (2017), 411--441.

\bibitem{Lebon1989}	Lebon, G., Cloot, A., {\it Propagation of ultrasonic sound waves in dissipative dilute gases and extended irreversible thermodynamics}, Wave Motion, {\bf 11} (1989), 23--32.

\bibitem{Lizama2019}	Lizama, C., Zamorano, S., {\it Controllability results for the Moore--Gibson--Thompson equation arising in nonlinear acoustics}, J. Diﬀerential Equations {\bf 266} (2019), 7813--7843.


\bibitem{Muller} M\"uller, I., Ruggeri, T., {\it Rational and Extended Thermodynamics}, Springer-Verlag, New York, 1998.

\bibitem{Nik4} Nikoli\'c, V., Said-Houari, B., {\it On the Jordan--Moore--Gibson--Thompson wave equation in hereditary fluids with quadratic gradient nonlinearity}, J. Math. Fluid Mech. {\bf 23} (2021), 1--24.

\bibitem{Nik5} Nikoli\'c, V., Said-Houari, B., {\it Mathematical analysis of memory effects and thermal relaxation in nonlinear sound waves on unbounded domains},  J. Diﬀerential Equations {\bf 273} (2021), 172--218.

\bibitem{Nik6}	Nikoli\'c, V., Said-Houari, B., {\it Asymptotic behavior of nonlinear sound waves in inviscid media with thermal and molecular relaxation}, Nonlinear Anal. Real World Appl. {\bf 62} (2021), 103384.

\bibitem{Nik3} Nikoli\'c, V., Winkler, M., {\it $L^{\infty}$ blow-up in the Jordan--Moore--Gibson--Thompson equation}, Nonlinear Anal. {\bf 247} (2024), 113600.

\bibitem{Quint4} Quintanilla, R., Racke, R., {\it Qualitative aspects in dual-phase-lag thermoelasticity}, SIAM J. Appl. Math. {\bf 66} (2006), 977--1001.

\bibitem{Quin} Quintanilla, R., Racke, R., {\it A note on stability in three-phase-lag heat conduction}, Int. J. Heat Mass Transf. {\bf 51} (2008), 24--29.

\bibitem{PQUI} Quittner, P., Souplet, P., {\it Superlinear parabolic problems. Blow-up, global existence and steady states}, Birkh\"auser Advanced Texts, Birkh\"auser, Basel/Boston/Berlin, 2007.

\bibitem{Racke2021}
Racke, R., Said-Houari, B., {\it Global well-posedness of the Cauchy problem for the 3D Jordan--Moore--Gibson--Thompson equation}, Commun. Contemp. Math. {\bf 23} (2021), 2050069.

\bibitem{Roy} Roy Choudhuri, S.K., {\it On a thermoelastic three-phase-lag model}, J. Therm. Stresses {\bf 30} (2007), 231--238.

\bibitem{Rudenko1977}	Rudenko, O.V., Soluyan, S.I., {\it Theoretical foundations of nonlinear acoustics}. Translated from the Russian by Robert T. Beyer. Studies in Soviet Science. Consultants Bureau, New York--London, (1977).

\bibitem{SaidHouari2022} Said-Houari, B., {\it Global well-posedness of the Cauchy problem for the Jordan--Moore--Gibson--Thompson equation with arbitrarily large higher-order Sobolev norms}, Discrete Contin. Dyn. Syst. {\bf 42} (2022),  4615--4635.

\bibitem{Simon-1996} Simon, J., {\it Compact sets in the space $L^p(0, T;B)$}, Ann. Mat. Pura Appl. {\bf 146} (1987), 65--96.

\bibitem{Souplet2002}
Souplet, P., {\it Gradient blow-up for multidimensional nonlinear parabolic equations with general boundary conditions}, Differential Integral Equations \textbf{15} (2002), 237--256.

\bibitem{Strau} Straughan, B., {\it Heat Waves}, Springer, New York, 2011.

\bibitem{Tzou1} Tzou, D.Y., {\it A unified approach for heat conduction from macro to micro-scales}, ASME J. Heat Transfer {\bf 117} (1995), 8--16.

\bibitem{Tzou2} Tzou, D.Y., {\it The generalized lagging response in small-scale and high-rate heating}, Int. J. Heat Mass Transf. {\bf 38} (1995), 3231--3240.

%
\end{thebibliography}
\end{document}